\documentclass{svmult}
\pdfoutput=1

\usepackage{type1cm}        

\usepackage{graphicx}        
\usepackage{multicol}        
\usepackage[bottom]{footmisc}

\usepackage{amsmath,amssymb}
\usepackage{bbm}
\usepackage{bm}
\usepackage{graphicx}
\usepackage[T1]{fontenc}
\usepackage[utf8]{inputenc}
\usepackage{url}
\usepackage{microtype}
\usepackage{import}
\usepackage{color}
\usepackage{epic}
\usepackage{tikz}
\usepackage{etex}
\usepackage{pgfplots}
\usepackage{pgfplotstable}
\usepackage{booktabs}
\usepackage{subfigure}
\usepackage{units}
\usepackage[ruled,vlined]{algorithm2e}

\begin{document}

\makeatletter
\newcommand{\dmn}{{\Omega}}
\newcommand{\dir}{{\Gamma_{D}}}
\newcommand{\neu}{{\Gamma_{N}}}
\newcommand{\rob}{{\Gamma_{R}}}
\newcommand{\norm}[1]{\left\lVert#1\right\rVert}
\newcommand\abs[1]{\left|#1\right|}
\newcommand{\DIV}[1]{\mbox{div}\left(#1\right)}
\newcommand{\TwoNorm}[2]{{\left\lVert#1\right\rVert}_{L^2(#2)}}
\newcommand{\VNorm}[2]{{\left\lVert#1\right\rVert}_{V,#2}}
\newcommand{\ddiv}{\operatorname{div}}
\newcommand{\supp}{\operatorname{supp}}
\newcommand{\G}{\mathcal{G}}
\newcommand{\N}{\mathcal{N}}
\newcommand{\Cor}{\mathcal{C}}
\newcommand{\cS}{\mathcal{S}}
\newcommand{\CIh}{C_{I_h}}
\newcommand{\CIH}{C_{I_H}}
\newcommand{\Col}{C_{\mathrm{ol}}}
\newcommand{\Colm}{C_{\mathrm{ol},m}}
\newcommand{\nei}{\mathsf{N}}
\newenvironment{rev}{\begin{color}{blue}}{\end{color}}
\makeatother

\title*{Multiscale Petrov-Galerkin Method for 
      High-Frequency Heterogeneous Helmholtz Equations}
\titlerunning{Multiscale Method for Heterogeneous Helmholtz Equation}

\author{Donald L. Brown
        \and Dietmar Gallistl
        \and Daniel Peterseim
}

\institute{
 Donald L. Brown
 \at
 School of Mathematical Sciences,
 The University of Nottingham,
 University Park,
 Nottingham, United Kingdom, \email{donald.brown(at)nottingham.ac.uk}
 \and
 Dietmar Gallistl
 \at
 Institut f\"ur Numerische Simulation, 
 Universit\"at Bonn, 
 Wegelerstr.~6,
 53115 Bonn, Germany, \email{gallistl(at)ins.uni-bonn.de}
 \and
 Daniel Peterseim
 \at
 Institut f\"ur Numerische Simulation, 
 Universit\"at Bonn, 
 Wegelerstr.~6,
 53115 Bonn, Germany, \email{peterseim(at)ins.uni-bonn.de}
}

\maketitle

\abstract{This paper presents a multiscale Petrov-Galerkin
finite element method for time-harmonic acoustic scattering problems with heterogeneous coefficients
in the high-frequency regime. We show that the method is pollution-free
also in the case of heterogeneous media
provided that the stability bound of the continuous problem grows at most 
polynomially with the wave number $k$.
By generalizing classical estimates of 
[Melenk, Ph.D.\ Thesis 1995] and
[Hetmaniuk, Commun.\ Math.\ Sci.\ 5 (2007)] for homogeneous medium, 
we show that this assumption of polynomially  wave number growth  holds
true for a particular class of smooth
heterogeneous material coefficients. Further, we present numerical examples to verify our stability estimates and implement an example
 in the  wider class of discontinuous coefficients to show computational applicability beyond our limited class of coefficients. 
}

\begin{acknowledgement}
The authors acknowledge the support given by the Hausdorff Center
for Mathematics Bonn. D.~Peterseim is supported by 
Deutsche Forschungsgemeinschaft in the Priority Program 1748 
``Reliable simulation techniques in solid mechanics. 
Development of non-standard discretization methods, mechanical 
and mathematical analysis'' under the project 
``Adaptive isogeometric modeling of propagating strong discontinuities
in heterogeneous materials''.
\end{acknowledgement}

\section{Introduction}\label{DONALD_BROWN:introduction}

The time-harmonic acoustic wave-propagation is customarily described
by the Helmholtz equation, which is of second-order, elliptic,
but indefinite. 
Its numerical solution therefore exhibits
severe difficulties especially in the regime of high wave numbers
$k$. 
It is well-known 
that the mesh size $h$ required for the stability of a standard
finite element method must be much smaller than a mesh size
$H$ which would be sufficient for a reasonable representation
of the solution. The phenomenon that the ratio $H/h$ tends to
infinity as $k$ grows, is known as the \emph{pollution effect}
\cite{DONALD_BROWN:Babuska:2000:PEF:354138.354150}.
A method is referred to as pollution-free, if $h$ and $H$ have the
same order of magnitude and so proper resolution of the solution
-- usually a certain fixed number of grid points per wave length --
implies quasi-optimality of the method.

When studying acoustic wave-propagation, it is often assumed to have constant material properties such as density and speed of sound, 
while in real complex materials, such as composites, these may be
heterogeneous.
Therefore, in this paper we study a multiscale Petrov-Galerkin method
for the Helmholtz equation with large wave numbers
$k$ and possibly heterogeneous material coefficients
as a generalization of
\cite{DONALD_BROWN:Gallistl.Peterseim:2015,DONALD_BROWN:Peterseim2014}.
Standard first-order piecewise polynomials on the scale $H$
serve as trial functions in this method, whereas the test functions
involve a correction by solutions to coercive cell problems on the 
scale $h$.
The size of the cells is proportional to $H$, where the proportionality
constant $m$ ---the oversampling parameter--- can be adjusted.
Typically $m\approx\log k$, depending on the stability of the 
problem, leads to a quasi-optimal method.
These local problems are translation invariant. Therefore, in
periodic media only a small number of corrector problems must
 be solved depending on the number of local mesh configurations.

The stability of the method requires that the stability constant
of the continuous operator depends polynomially on $k$. 
Such results are very rare in the literature even for the case of homogeneous media. 
We shall emphasize that such an assumption does not hold true in general \cite{DONALD_BROWN:betcke}.
The first positive estimates of this type go back to
\cite{DONALD_BROWN:melenk_phd} for convex planar domains with pure
Robin boundary. They were later generalized to other settings
and three spatial dimensions in
\cite{DONALD_BROWN:feng,DONALD_BROWN:hetmaniuk}.
For instance, in the particular case of pure impedance boundary conditions with
$\partial\Omega=\Gamma_R$, it was proved in 
\cite{DONALD_BROWN:feng,DONALD_BROWN:MelenkEsterhazy,DONALD_BROWN:melenk_phd}, 
by employing a technique of \cite{DONALD_BROWN:makridakis}, that the inf-sup constant is bounded, i.e.
$\gamma(k,\Omega,A,V^2)\lesssim k$.
Further setups allowing for polynomially well-posedness in the presence of a single star-shaped sound-soft scatterer are described in
\cite{DONALD_BROWN:hetmaniuk}. 
For multiple scattering and, in particular, for scattering in heterogeneous media, the situation is completely open. 
To show that the assumption is satisfiable  for non-trivial heterogeneous media, in this work we determine a class of smooth heterogeneous coefficients
that allow for explicit-in-$k$ stability estimates. 

\subsection{Heterogeneous Helmholtz Problem}

We begin with some standard notation on complex-valued Lebesgue and Sobolev spaces
that applies throughout this paper. 
The bar indicates
complex conjugation and $i$ is the imaginary unit.
The $L^2$ inner product is denoted
by 
$(v,w)_{L^2(\Omega)}:=\int_\Omega v\bar w\,dx$. 
The Sobolev space of complex-valued $L^p$ functions over a
domain $\omega$ whose 
generalized derivatives up to order $l$ belong to $L^p$ is
denoted by $W^{l,p}(\omega)$ and $H^l(\omega):=W^{l,2}(\omega)$.
Further, the notation $A\lesssim B$ abbreviates $A\leq C B$ for some
constant $C$ that is independent of the mesh-size,
the wave number $k$, and all further parameters in the method like
the oversampling parameter $m$ or the fine-scale mesh-size $h$;
$A\approx B$ abbreviates $A\lesssim B\lesssim A$.

 We now begin with some notation and problem setting. Let $\Omega\subset \mathbb{R}^d$ be an open  bounded Lipschitz domain with polyhedral boundary for $d\in\{1,2,3\}$.  
We wish to find a solution $u$ that satisfies 
\begin{align}\label{DONALD_BROWN:main}
 -\DIV{A(x) \nabla u}-k^2V^2(x)u&=f \text{ in } \dmn,
\end{align}
along with the boundary conditions
\begin{subequations}\label{DONALD_BROWN:BCs}
\begin{align}
 \label{DONALD_BROWN:dirichlet}
  u&=0 \text{ on } \dir,\\
 \label{DONALD_BROWN:neumann}
A(x)\nabla u \cdot \nu &=0 \text{ on } \neu ,\\
 \label{DONALD_BROWN:robin}
A(x)\nabla u \cdot \nu -i k\beta(x)u&=g \text{ on } \rob. 
\end{align}
\end{subequations}
Here, $\nu$ denotes the outer normal to  $\partial \Omega=\overline{\dir \cup \neu \cup \rob}$, where the boundary sections are assumed disjoint. We suppose that $|\rob|>0$, but allow the other portions of the boundary to have measure zero. 
Although the results in this paper hold for a weaker dual space here we suppose $f\in L^2({\Omega})$ and   $g\in L^2({\rob})$.
For the coefficients, we suppose
$A(x), V^2(x)\in W^{1,\infty}(\dmn)$, and  $\beta(x) \in L^{\infty}(\dmn)$ are real valued. 
Moreover, we suppose there exist positive  constants $A_{min}, A_{max}, \beta_{min},\beta_{max},V_{min}$, and $V_{max}$ independent of $k$  such that for almost all $x\in\dmn$ we have
\begin{subequations}\label{DONALD_BROWN:upperbounds}
\begin{align}
A_{min} \leq&A(x)\leq A_{max},\\
\beta_{min}\leq&\beta(x)\leq \beta_{max},\\
V^2_{min} \leq&V^2(x)\leq V^2_{max}.
\end{align}
\end{subequations}
We denote the space 
$$V:=\{u\in H^1(\dmn)\;|\;u=0 \text{ on }  \dir\}$$
 and denote the  norm weighted with $A(x),V(x),$ and $k$ to be for $\omega\subset\Omega$
\begin{align}
\VNorm{u}{\omega}:=\sqrt{ \TwoNorm{k V u}{\omega}^2+\TwoNorm{A^{\frac{1}{2}}\nabla u}{\omega}^2},
\end{align}
where if $\omega=\Omega,$ we simply write $\norm{u}_{V}$.
We have the following variational form corresponding to \eqref{DONALD_BROWN:main}: Find $u\in V$ such that  
\begin{align}\label{DONALD_BROWN:var}
a(u,v)=(f,v)_{L^2(\dmn)}+(g,v)_{L^2(\rob)}\text{ for all }v\in V,
\end{align}
where the complex-valued sesquilinear form $a : V\times V\to \mathbb{C}$ is given by 
\begin{align}\label{DONALD_BROWN:varform}
a(u,v)=(A(x) \nabla u,\nabla v)_{L^2(\dmn)}-(k^2 V^2(x) u ,v)_{L^2(\dmn)}- (i k \beta(x) u,v)_{L^2(\rob)}.
\end{align}
Here we write $(u,v)_{L^2(\dmn)}=\int_{\dmn}u \bar{v}dx$ and similarly $(u,v)_{L^2(\rob)}=\int_{\rob}u \bar{v}ds$.

\subsection{Motivation for a Multiscale Method and Stability Analysis}

It is well known
\cite{DONALD_BROWN:Babuska:2000:PEF:354138.354150}
that the pollution effect cannot be avoided in standard methods.
However, it may be overcome by coupling the polynomial degree
of the method with the wave number $k$ 
\cite{DONALD_BROWN:mm_stas_helm3}.
Therefore, multiscale methods appear to be a natural tool to 
incorporate fine-scale features in a low-order discretization.
Moreover, the parameters of this method must be coupled
logarithmically with the wave number and therefore require the stability constant of the continuous problem
to be polynomially dependent of $k$ to arrive at a computationally efficient method.
Hence,
the stability of the continuous heterogeneous problem \eqref{DONALD_BROWN:main} is critical to the analysis of the related algorithms. 
In general, it is often shown  (or possibly assumed) that there
exists some constant  $C_{stab}(k,\Omega,A,V^2)>0,$ which depends on $k$, the geometry, and the coefficients,  such that 
\begin{equation}\label{DONALD_BROWN:upperstabilityestimate}
\norm{u}_V\leq C_{stab}(k,\Omega,A,V^2)\left(\norm{f}_{L^2(\Omega)}+\norm{g}_{L^2(\rob)} \right).
\end{equation}
Further, turning to the inf-sup type lower bound, it is often shown,  or possibly assumed, that there
exists some constant $\gamma(k,\Omega,A,V^2)$, related to $C_{stab}(k,\Omega,A,V^2)$,  such that 
\begin{equation}\label{DONALD_BROWN:e:helmholtzstability}
\gamma(k,\Omega,A,V^2)^{-1}
  \leq
  \inf_{v\in V\setminus\{0\}} \sup_{w\in V\setminus\{0\}} 
   \frac{\operatorname{Re} a(v,w)}{\|v\|_V\|w\|_V} .
\end{equation}
As noted, it is often the case that these constants depend merely polynomially on $k$. 
However, it has been demonstrated that there are special instances of exponential $k$ dependence on $C_{stab}(k,\Omega,A,V^2)$  \cite{DONALD_BROWN:betcke}, and thus, highly unstable inf-sup constants 
$\gamma(k,\Omega,A,V^2)$. 

\section{Stability of the Heterogeneous Helmholtz Model}\label{DONALD_BROWN:stability}


As discussed in Section \ref{DONALD_BROWN:introduction}, the stability and regularity of the continuous problem has been investigated  for constant coefficients in various contexts.
In this section, we shall investigate the stability of the continuous problem with respect to wave number in the case of heterogeneous coefficients. 
We proceed using the variational techniques
 with geometric constraints \cite{DONALD_BROWN:hetmaniuk}.

As noted in Section \ref{DONALD_BROWN:introduction}, in the case of 
constant coefficients, there exist various methods to bound 
$\gamma(k,\Omega,A,V^2)$ from \eqref{DONALD_BROWN:e:helmholtzstability} 
in terms of $k$. 
Most importantly, the possible exponential dependence discussion in  \cite{DONALD_BROWN:betcke}, will be 
excluded here.
We will show in this section, that for certain classes of coefficients, we are able to obtain a favorable polynomial
bound for $\gamma(k,\Omega,A,V^2)$. To this end, we will employ variational techniques and so-called Rellich type identities with restrictions on the types of geometries similar to 
the work of \cite{DONALD_BROWN:hetmaniuk} and references therein.

As we use the variational techniques we will 
 make the geometric assumptions made by \cite{DONALD_BROWN:hetmaniuk}. That is we suppose that there exists a $x_{0}\in \mathbb{R}^d$ and a $\eta>0$  such that  
\begin{subequations}\label{DONALD_BROWN:GeomAssum}
\begin{align}
(x-x_{0})\cdot \nu &\leq 0 \text{ on } \dir,\\
(x-x_{0})\cdot \nu& = 0 \text{ on } \neu,\\
(x-x_{0})\cdot \nu &\geq \eta \text{ on } \rob.
\end{align}
\end{subequations}
For a summary of such possible domains, we refer the reader to \cite{DONALD_BROWN:hetmaniuk}. However, to get some sense of a geometry the reader may envision a convex domain with pure 
 impedance boundary conditions. This of course may be weakened. 

\subsection{Statement of Stability, Connections to Inf-Sup Constants, and Boundedness}

In this section we present our main stability result. The variational techniques employed require assumptions on the class of coefficients to remain valid. 
We outline these constraints and obtain a bounded-in-$k$ result. We further relate these to the inf-sup constants and explore the boundedness of the non-constant coefficient case. 

\begin{theorem}\label{DONALD_BROWN:Theorem1}
Suppose $\Omega\subset \mathbb{R}^d$, is a bounded connected Lipschitz domain and   satisfies the geometric  assumptions \eqref{DONALD_BROWN:GeomAssum}. 
Let $u$ be a solution of \eqref{DONALD_BROWN:main} with the boundary conditions \eqref{DONALD_BROWN:BCs}, coefficients  satisfying the bounds \eqref{DONALD_BROWN:upperbounds}, and $k\geq k_{0}>0$, for some $k_{0}$.
Further, we suppose the regularity $u\in H^{3/2+\delta}(\dmn)$ 
for some $\delta>0$.

Define the following function
\begin{align}\label{DONALD_BROWN:Sfunction.theorem}
S(x):=\ddiv\left(\left(\frac{V^2(x)}{A(x)}\right) (x-x_{0})\right)
\end{align}
and further, we will denote $C_{G}$ to be the minimal constant so that 
\begin{align}
2 \abs{\int_{\dmn}\left(\frac{\nabla A}{A}\right)\nabla u ((x-x_{0}) \cdot \nabla \bar{u}) dx}\leq C_{G}\norm{\left(\frac{\nabla A}{A}\right)}_{L^{\infty}(\dmn)}\TwoNorm{\nabla u}{\dmn}^2.
\end{align}
We suppose that
\begin{subequations}  \label{DONALD_BROWN:conditions}
\begin{align}
\label{DONALD_BROWN:conditions.1}
&S_{min}=\min_{x\in \dmn}S(x)>0, \\
\label{DONALD_BROWN:conditions.2}
&S_{min}-\left((d- 2) +C_{G}\norm{\left(\frac{\nabla A}{A}\right)}_{L^{\infty}(\dmn)}\right)\frac{V_{max}^2}{A_{min}}>0.
\end{align}
\end{subequations}
We then have the following estimate
\begin{align}\label{DONALD_BROWN:mainestimate.theorem}
\norm{u}^2_V \leq C^*\left(1+\frac{1}{k^2}\right)\left(\TwoNorm{f}{\dmn}^2+\TwoNorm{g}{\rob}^2\right),
\end{align}
where $C^*$ depends only on the \eqref{DONALD_BROWN:upperbounds} and $\dmn$, but not on $k$.
\end{theorem}

\begin{proof}
See Appendix below.\qed
\end{proof}

\begin{remark}
The assumption from Theorem~\ref{DONALD_BROWN:Theorem1} that $u$
satisfy the regularity $u\in H^{3/2+\delta}(\dmn)$ is an assumption
on the configuration of the boundary decomposition into
$\Gamma_D$, $\Gamma_N$, $\Gamma_R$. It is not a further
restriction on the coefficients $A$ or $V^2$.
\end{remark}

Now that we have an explicit bound for a class of constant variable coefficients, we now will relate the constant $C_{stab}(k,\Omega,A,V^2):=C^*\left(1+\frac{1}{k^2}\right)$ 
to $\gamma(k,\Omega,A,V^2)$ given by  \eqref{DONALD_BROWN:e:helmholtzstability}.

\begin{theorem}\label{DONALD_BROWN:lowerstability}
Supposing the assumptions in Theorem \ref{DONALD_BROWN:Theorem1}, 
 we have the following estimate
\begin{equation}\label{DONALD_BROWN:lowerstabilityestimate}
k^{-1} \lesssim \widetilde\gamma^{-1}
  \lesssim 
  \inf_{v\in V\setminus\{0\}} \sup_{w\in V\setminus\{0\}} 
   \frac{\operatorname{Re} a(v,w)}{\|v\|_V\|w\|_V} .
\end{equation}
Where, $\widetilde\gamma:=(1+  C^*\left(k+\frac{1}{k}\right)V_{max}^2)$.
\end{theorem}
\begin{proof}
We proceed by a standard argument from \cite{DONALD_BROWN:MelenkEsterhazy}, adapted to the heterogeneous case. Given  $u\in H^1(\Omega)$,  define $z\in H^1(\Omega)$ as the solution of 
\begin{align}\label{DONALD_BROWN:zequation}
2k^2( v,V^2 u)_{L^2(\dmn)}=a(v,z), \text{ for all } v\in V.
\end{align}
Then, from the estimate \eqref{DONALD_BROWN:mainestimate.theorem}, we have 
\begin{align}\label{DONALD_BROWN:zestimate}
\norm{z}_{V}
\leq
 C^*\left(1+\frac{1}{k^2}\right) V_{max}^2 k^2 \norm{u}_{L^2(\Omega)} .
\end{align}
Note that 
$$
\operatorname{Re} a(u,u)=(A(x) \nabla u,\nabla u)_{L^2(\dmn)}-(k^2 V^2(x) u ,u)_{L^2(\dmn)}
$$
and using \eqref{DONALD_BROWN:zequation} and taking $v=u+z$ implies
\begin{align}
\operatorname{Re} a(u,v)=\operatorname{Re} a(u,u) +\operatorname{Re} a(u,z)=\norm{u}^2_{V}.
\end{align}
Using \eqref{DONALD_BROWN:zestimate} we obtain
 \begin{align*}
\norm{v}_{V}\leq & \norm{u}_{V}+\norm{z}_{V}\leq \norm{u}_{V}+C^*\left(1+\frac{1}{k^2}\right)V_{max}^2 k^2 \norm{u}_{L^2(\Omega)}\\
&\leq (1+  C^*\left(k+\frac{1}{k}\right)V_{max}^2)\norm{u}_{V}.
 \end{align*}
 Hence, $\operatorname{Re} a(u,v)=\norm{u}_{V}^2
  \geq 
(1+  C^*\left(k+\frac{1}{k}\right)V_{max}^2)^{-1}\norm{v}_{V}\norm{u}_{V},$ taking 
 $$\widetilde\gamma:=(1+  C^*\left(k+\frac{1}{k}\right)V_{max}^2)\approx k^{}$$ yields the result.
  \qed
 
\end{proof}
Finally, for completeness, we include a brief proof of the boundedness of the variational from.
\begin{theorem}\label{DONALD_BROWN:boundednesstheorem}
Supposing the assumptions in Theorem \ref{DONALD_BROWN:Theorem1}, 
the variational form \eqref{DONALD_BROWN:varform} has the following boundedness property
\begin{align}\label{DONALD_BROWN:boundedform}
|a(u,v)|\leq C_{a} \norm{u}_{V}\norm{v}_{V}.
\end{align}
Here $C_{a}$ may depend on the bounds \eqref{DONALD_BROWN:upperbounds}, multiplicative trace constants, and $\dmn$,  but not $k$. 
\end{theorem}
\begin{proof}
From the variational form we have
\begin{align*}
|a(u,v)| & \leq \Big|(A^{\frac{1}{2}} \nabla u,A^{\frac{1}{2}}\nabla v)_{L^2(\dmn)}\Big|+\left |(k V u ,k V v)_{L^2(\dmn)} \right |
\\
 &\qquad\qquad+\left |((\beta k)^{\frac{1}{2}} u,(\beta k )^{\frac{1}{2}} v)_{L^2(\rob)}\right|&\\
 &\leq \TwoNorm{A^{\frac{1}{2}}\nabla u}{\dmn} \TwoNorm{A^{\frac{1}{2}}\nabla v}{\dmn}+\TwoNorm{k V u}{\dmn}\TwoNorm{k V v}{\dmn}\\
 &\qquad\qquad+\TwoNorm{(\beta k )^{\frac{1}{2}} u}{\rob}\TwoNorm{(\beta k )^{\frac{1}{2}} v}{\rob}&\\
 &\lesssim\norm{u}_V \norm{v}_{V}+\TwoNorm{(\beta k )^{\frac{1}{2}} u}{\rob}\TwoNorm{(\beta k )^{\frac{1}{2}} v}{\rob}
 .
\end{align*}
We have from the   multiplicative trace inequality 
\begin{align*}
 \TwoNorm{ k^\frac{1}{2}u}{\rob}^2& \leq C_{M}\left(\TwoNorm{ k^\frac{1}{2}u}{\dmn} \left| k^\frac{1}{2}u\right |_{H^1(\dmn)}+\text{diam}(\dmn)^{-1}\TwoNorm{ k^\frac{1}{2}u}{\dmn}^2 \right)\\
& \leq C_{M}\left(\TwoNorm{ ku}{\dmn}^{2}+ \left| u\right |^2_{H^1(\dmn)}+\text{diam}(\dmn)^{-1}\TwoNorm{ k^\frac{1}{2}u}{\dmn}^2 \right)\\
&\lesssim C_{M}\left(\norm{u}_{V}+\text{diam}(\dmn)^{-1}\TwoNorm{ k u}{\dmn}^2 \right)\lesssim C_{M}\norm{u}_{V}^2
\end{align*}
since $k\geq 1$. Applying this to the $\rob$ terms we arrive at \eqref{DONALD_BROWN:boundedform}.
  \qed
\end{proof}

\subsection{Example Coefficients}

In this subsection, we will provide a few examples that satisfy the assumptions on the coefficients \eqref{DONALD_BROWN:conditions}.
Hence,  the set of 
 bounded smooth coefficients that yields polynomial-in-$k$ bounds is non-trivial. 
 We show that for some coefficients, as the oscillations become more frequent we violate the conditions \eqref{DONALD_BROWN:conditions}. In particular, it appears that the restriction
 on the 
 amplitude of the coefficients is related to the restrictions on the frequency of oscillations.

To simplify things, yet provide non-trivial coefficients, we will only consider radially symmetric conditions in $\mathbb{R}^2$.
Indeed, even with this symmetry, we are able to highlight the complexities and restrictiveness in these conditions. We will see that the frequency of oscillations play a considerable role
in violation of these conditions, as well as the amplitude.
 
We take  $\dmn\subset\mathbb{R}^2$ to be given by the unit circle
$\dmn: =\{(x,y)\in \mathbb{R}^2 \ | \ x^2+y^2\leq1 \}$ and $\partial \dmn =\{(x,y)\in \mathbb{R}^2\ |  \ x^2+y^2=1 \}$. 
Further, we will take  $\neu=\dir=\emptyset,$ so that $\rob=\partial \dmn$.  
We take $x_{0}=(0,0)\in \dmn$, and so
$m=(x-x_{0})=r\hat{r}$, where $r^2=x^2+y^2$ and $\hat{r}$ is the standard unit normal in radial coordinates.
Then, clearly, 
$m \cdot \nu =1$ on $\rob$ and so the geometric assumptions \eqref{DONALD_BROWN:GeomAssum} are satisfied with this domain. 
We will take $\beta(x)=1$, $g(x)=0$, and suppose that $f:=f(r)$,
is a given radially symmetric forcing. 
We finally suppose that the heterogeneities are radially symmetric, 
$V^2(x)=V^2(r),$ and $A(x)=A(r).$ 
We briefly recall in radial coordinates that for a function $A$
and a vector field $\sigma=(\sigma_{r},\sigma_{\theta})$
\begin{align*}
\mbox{div}( \sigma )&=\frac{1}{r}\frac{\partial }{\partial r}(r \sigma_r)+\frac{1}{r}\frac{\partial  \sigma_\theta }{\partial \theta}.\\
\nabla A& =\frac{\partial  A }{\partial r}\hat{r}+\frac{1}{r}\frac{\partial  A }{\partial \theta}\hat{\theta}.\\
\int_{\dmn}A dxdy&=\int^{2 \pi}_{0}\int^{1}_{0}A rdrd\theta,
\end{align*}
where $\hat{\theta}$ is the standard angular coordinate.
By examining the conditions \eqref{DONALD_BROWN:conditions}, we are able to produce a few interesting examples.

{\bf Case 1: $A=1$.} 
Note that from condition \eqref{DONALD_BROWN:conditions.2}, that if $A=1$ (or constant), we see that the conditions simplify slightly since the gradient terms in $A$ will vanish. 
Indeed, now we see that only condition  \eqref{DONALD_BROWN:conditions.1} must be satisfied. 
In this setting,  we must have that $\mbox{div}(V^2 m)>0$ for our estimates to hold, or  rewritten in radial coordinates as 
\begin{align}\label{DONALD_BROWN:cond.ex}
 \frac{1}{r}\frac{\partial }{\partial r} \left(V^2(r) r^2\right)>0.
\end{align}
From this condition we may choose a few possible coefficients for $V(r)$. 
A trivial example is when $V^2(r)=r+1$. Clearly, $$\frac{1}{r}\frac{\partial }{\partial r} \left( r^3+r^2\right)=\frac{1}{r}(3r^2+2r)=3r+2>0.$$
Many such polynomial in $r$ choices exist as long as they do not violate boundedness and positivity.

More interesting examples come from  oscillatory coefficients. 
Suppose, for $\epsilon>0$,  we take now the innocent looking example
\begin{align}\label{DONALD_BROWN:numcoef1}
V^2(r)=\frac{1}{2}\sin\left(\frac{2\pi r}{\epsilon}\right)+5,
\end{align}
and so 
\begin{align}
 \frac{1}{r}\frac{\partial }{\partial r} \left(\frac{r^2}{2}\sin\left(\frac{2\pi r}{\epsilon}\right)+5r^2\right)
 =\sin\left(\frac{2\pi r}{\epsilon}\right)+\frac{r\pi}{\epsilon}\cos\left(\frac{2\pi r}{\epsilon}\right)+10.
\end{align}
A quick investigation shows that if $\epsilon=1,$ then \eqref{DONALD_BROWN:cond.ex} is satisfied, however, when $\epsilon=.1$ it is violated. 
Hence, if the coefficient becomes highly oscillatory, the stability condition is not satisfied. Also note that if we fix $\epsilon=1$,
but extend the domain from a unit circle to one of radius $R$, we will eventually enter a negative region. Hence, the domain size also may have an effect on 
stability from the viewpoint of conditions \eqref{DONALD_BROWN:conditions}.

{\bf Case 2: $A=V^2$.}
Turning to the definition of $S(x)$ in \eqref{DONALD_BROWN:Sfunction.theorem}, we see that if $A=V^2$, the functions simplifies to  $S(x)=d$. 
Thus, condition \eqref{DONALD_BROWN:conditions.1} is  always satisfied. For $d=2$, \eqref{DONALD_BROWN:conditions.2} becomes
\begin{align}\label{DONALD_BROWN:condition.2.example}
 2-\left(C_{G}\norm{\left(\frac{\nabla A}{A}\right)}_{L^{\infty}(\dmn)}\right)\frac{A_{max}}{A_{min}}>0.
\end{align}
Taking a closer look at the terms related to $C_G$ 
from Theorem~\ref{DONALD_BROWN:Theorem1}, we have in radial coordinates
\begin{align*}
&2 \abs{\int_{\dmn}\left(\frac{\nabla A}{A}\right)\nabla u ((x-x_{0}) \cdot \nabla \bar{u}) dx}
\\
&\qquad
=2 \abs{\int_{0}^{2\pi}\int_{0}^1 \left(\frac{r^2}{A(r)}\frac{\partial A(r)}{\partial r}\right) 
\left|
\frac{\partial u(r)}{\partial r}\right|^2 drd\theta}\\
&\qquad
\leq 2 \norm{\frac{1}{A(r)}\frac{\partial A(r)}{\partial r}}_{L^{\infty}(\dmn)}  \norm{\nabla u }_{L^2(\dmn)}^2.
\end{align*}
Hence, we may take here $C_{G}=2$. Noting that $$\frac{\partial}{\partial r} \ln(A)=\frac{1}{A(r)}\frac{\partial A(r)}{\partial r},$$  
then the condition \eqref{DONALD_BROWN:condition.2.example} becomes 
\begin{align}
 1-\left(\norm{\frac{\partial}{\partial r} \ln(A)}\right)\frac{A_{max}}{A_{min}}>0.
\end{align}
Taking 
\begin{align}\label{DONALD_BROWN:numcoef2}
 V^2(r)=A(r)=\exp\left(\alpha \left(\sin\left(\frac{r}{\epsilon}\right)+\delta \right)\right),
\end{align}
for $\epsilon,\alpha,$ and $\delta$ positive,
then 
$$
\norm{\frac{\partial}{\partial r} \ln(A)}_{L^{\infty}(\dmn)}=\norm{\frac{\alpha}{\epsilon}\cos\left(\frac{r}{\epsilon}\right)}_{L^{\infty}(\dmn)}=\frac{\alpha}{\epsilon}.
$$
Note further that $A_{max}=\exp(\alpha(\delta+1))$ and $A_{min}=\exp(\alpha(\delta-1))$, and so $\frac{A_{max}}{A_{min}}=\exp(2\alpha)$. Hence, 
\begin{align}
 1-\left(\norm{\frac{\partial}{\partial r} \ln(A)}\right)\frac{A_{max}}{A_{min}}=1-\frac{\alpha}{\epsilon}\exp(2\alpha) >0
\end{align}
or $\alpha \exp(2 \alpha)< \epsilon.$ 
We see from this calculation that the frequency of oscillation in the coefficients is related to the amplitude as far as the conditions \eqref{DONALD_BROWN:conditions} are concerned. 
The more oscillatory the function, the smaller the amplitude must be in this example. 

\section{The Multiscale Method}\label{DONALD_BROWN:s:method}
In this  section, we will introduce the notation on finite element spaces and
meshes that  define the multiscale Petrov-Galerkin method 
(msPGFEM) for the heterogeneous  Helmholtz problem. 
This method is based on ideas in an  algorithm developed for
homogenization problems in
\cite{DONALD_BROWN:MP14,DONALD_BROWN:HP12,DONALD_BROWN:Brown.Peterseim:2014}
also known as Localized Orthogonal Decomposition.
The ideas have been adapted to the Helmholtz problem for
homogeneous coefficients in \cite{DONALD_BROWN:Peterseim2014},
and later presented in the Petrov-Galerkin framework  
\cite{DONALD_BROWN:Gallistl.Peterseim:2015,DONALD_BROWN:Peterseim2015survey}.
We will stay in line with the notation and presentation of \cite{DONALD_BROWN:Gallistl.Peterseim:2015},
as this is the basis for the algorithm applied to a heterogeneous medium. 
We begin by defining the basic components needed, then define the multiscale method as well as some computational aspects. Finally, we will briefly discuss the
error analysis for the method, however, this will not differ too far from the homogeneous coefficient algorithm and as thus, will refer the reader to
technical proofs in \cite{DONALD_BROWN:Gallistl.Peterseim:2015}.

\subsection{Meshes and Data Structures}
We begin with the basic notation needed regarding the relevant mesh and data structures.
For the sake of clarity and completeness, 
we will briefly recall the notation used in \cite{DONALD_BROWN:Gallistl.Peterseim:2015}.
Let $\G_H$ be a regular partition of $\Omega$ into 
intervals, parallelograms, parallelepipeds for
$d=1,2,3$, respectively, such that $\bigcup\G_H =\overline\Omega$
and any two distinct $T,T'\in\G_H$ are either disjoint or share
exactly one lower-dimensional hyper-face
(that is a vertex or an edge for $d\in\{2,3\}$ or a face
for $d=3$).
We suppose the mesh is quasi-uniform.
For simplicity, we are considering quadrilaterals (resp. hexahedra) with
parallel faces, this guarantees the non-degeneracy of the elements in $\G_H$.
Again, the theory of this paper carries over to
partitions satisfying suitable non-degeneracy conditions
or even to meshless methods based on proper partitions of unity
\cite{DONALD_BROWN:HMP14}.

Given any subdomain $S\subseteq\overline\Omega$, we define its neighborhood 
to be
\begin{equation*}
\nei(S):=\operatorname{int}
          \Big(\cup\{T\in\G_H\,:\,T\cap\overline S\neq\emptyset  \}\Big).
\end{equation*}
Furthermore, we introduce for any $m\geq 2$ the patch extensions
\begin{equation*}
\nei^1(S):=\nei(S)
\qquad\text{and}\qquad
\nei^m(S):=\nei(\nei^{m-1}(S)) .
\end{equation*}
Note that the shape-regularity implies that there is a uniform bound denoted
$\Colm$,
on the number of elements in the $m$th-order patch,
$
\#\{ K\in\G_H\,:\, K\subseteq \overline{\nei^m(T)}\}
\leq \Colm
$
for all ${T\in\G_H}$.
We will abbreviate $\Col:=C_{\mathrm{ol},1}$.
The assumption that the coarse-scale mesh $\G_H$
is quasi-uniform implies that $\Colm$ depends polynomially
on $m$.
The global mesh-size is 
$H:=\max\{\operatorname{diam}(T)\}$ for all $T\in\G_H$.

We will denote $Q_p(\G_H)$ to be  the space of piecewise polynomials of partial
degree less than or equal to $p$.
The space of globally continuous piecewise first-order polynomials is given by
$
\cS^1(\G_H):= C^0(\Omega)\cap Q_1(\G_H),
$
and by incorporating the Dirichlet condition we arrive at the standard $Q_1$ finite element space denoted here as 
\begin{equation*}
V_H:=\cS^1(\G_H) \cap V.
\end{equation*}
The set of free vertices, or the degrees of freedom, is denoted by
$$
  \N_H:=\{z\in\overline\Omega\,:\, 
           z\text{ is a vertex of }\G_H\text{ and }z\notin\Gamma_D\}.
$$

To construct our fine-scale and, thus, multiscale spaces we will need to define a coarse-grid quasi-interpolation operator. 
For simplicity of presentation,we suppose here that this quasi-interpolation is also projective.
This assumption may be lifted c.f. \cite{DONALD_BROWN:HMP14} and references therein.
We let  $I_H:V\to V_H$ be a surjective
quasi-interpolation operator that
acts as a stable quasi-local projection in the sense that
$I^2_H = I_H$ and that
for any $T\in\G_H$ and all $v\in V$ the following local stability result holds
\begin{equation}\label{DONALD_BROWN:e:IHapproxstab}
H^{-1}\|v-I_H v\|_{L^2(T)} + \|\nabla I_H v \|_{L^2(T)}
\leq \CIH \|\nabla v\|_{L^2(\nei(T))} .
\end{equation}
Under the mesh condition that 
$$k H \lesssim 1$$
is bounded by a generic constant,
this implies  stability in the $\|\cdot\|_V$ norm
\begin{equation}\label{DONALD_BROWN:e:IHapproxstabV}
\|I_H v\|_V \leq C_{I_H,V} \|v\|_V
\quad\text{for all } v\in V,
\end{equation}
with a $k$-independent constant $C_{I_H,V}$. 
However, $C_{I_H,V}$, will depend on the constants in \eqref{DONALD_BROWN:upperbounds}.

One possible choice and which we use in our implementation of the method,
is to define $I_H:=E_H\circ\Pi_H$, where
$\Pi_H$ is the piecewise $L^2$ projection onto $Q_1(\G_H)$
and $E_H$ is the averaging operator that maps $Q_1(\G_H)$ to $V_H$ by
assigning to each free vertex the arithmetic mean of the corresponding
function values of the neighbouring cells, that is, for any $v\in Q_1(\G_H)$
and any free vertex $z\in\N_H$,
\begin{equation*}
(E_H(v))(z) =
           \sum_{\substack{T\in\G_H\\\text{with }z\in T}}v|_T (z) 
           \bigg/
           \#\{K\in\G_H\,:\,z\in K\}.
\end{equation*}
Note that with this choice of quasi-interpolation, $E_H(v)|_{\Gamma_D} = 0$ by construction.
For this choice,
the proof of \eqref{DONALD_BROWN:e:IHapproxstab} follows from combining the
well-established approximation and stability properties of 
$\Pi_H$ and $E_H$ shown in \cite{DONALD_BROWN:ern}.

\subsection{Definition of the Method}\label{DONALD_BROWN:ss:defMethod}

The multiscale method is determined by three parameters,
namely the coarse-scale mesh-size $H$, the fine-scale mesh-size $h$, and the oversampling parameter $m$. 
We assign to any $T\in\G_H$ its $m$-th order patch
$\Omega_T:=\nei^m(T)$, $m\in \mathbb{N}$, 
and define for any $v,w\in V$ the localized sesquilinear forms of  
\eqref{DONALD_BROWN:varform} to $\Omega_T$ as 
\begin{equation*}
\begin{aligned}
&a_{\Omega_{T}}(u,v)
\\ 
&\quad
=(A(x) \nabla u,\nabla v)_{L^2(\Omega_{T})}-(k^2 V^2(x) u ,v)_{L^2(\Omega_{T})}- (i k \beta(x) u,v)_{L^2(\rob\cap \partial \Omega_{T})}.
\end{aligned}
\end{equation*}
and to $T$, we have
\begin{equation*}
a_{\Omega_{T}}(u,v)=(A(x) \nabla u,\nabla v)_{L^2({T})}-(k^2 V^2(x) u ,v)_{L^2({T})}- (i k \beta(x) u,v)_{L^2(\rob\cap \partial {T})}.
\end{equation*}
Let the fine-scale mesh $\G_h,$ be a global uniform refinement of the mesh $\G_H$ over
$\Omega$ and define
\begin{equation*}
V_h(\Omega_T) 
 := \{ v\in Q_1(\G_h) \cap V\,: v=0\text{ outside }\Omega_T\} .
\end{equation*}
Define the null space
\begin{equation*}
W_h(\Omega_T) := \{ v_h\in V_h(\Omega_T) \,:\, I_H(v_h) = 0\}
\end{equation*}
of the quasi-interpolation operator $I_H$ defined in the previous section. This is the space often referred to as the fine-scale or small-scale space. 
Given any nodal basis function $\Lambda_z\in V_H$,
let 
$\lambda_{z,T}\in W_h(\Omega_T)$
solve the subscale corrector problem
\begin{equation}\label{DONALD_BROWN:e:lambdacorrectorproblem}
a_{\Omega_T}(w,\lambda_{z,T}) = a_T(w,\Lambda_z)
\quad\text{for all } w\in W_h(\Omega_T).
\end{equation}
Let $\lambda_z:=\sum_{T\in\G_H} \lambda_{z,T}$
and define the multiscale test function
\begin{equation*}
\widetilde\Lambda_z := \Lambda_z -  \lambda_z.
\end{equation*}
The space of multiscale test functions then reads
\begin{equation*}
\widetilde V_H := \operatorname{span}\{\widetilde\Lambda_z\,:\,z\in\N_H\} .
\end{equation*}
We emphasize that the dimension of the multiscale space is the same as the original coarse space,
$\dim V_H = \dim \widetilde V_H$. Moreover, it is independent of the parameters $m$ and $h$.
Finally, the multiscale Petrov-Galerkin FEM seeks to find  $u_H\in V_H$ such that
\begin{equation}\label{DONALD_BROWN:e:discreteproblem}
a(u_H,\tilde v_H) = 
(f,\tilde v_H)_{L^2(\Omega)}
  + (g,\tilde v_H)_{L^2(\Gamma_R)}
\quad\text{for all } \tilde v_H\in \widetilde V_H.
\end{equation}

As in \cite{DONALD_BROWN:Gallistl.Peterseim:2015}, the error analysis and the numerical experiments will show that
the choice $H\lesssim k^{-1}$, $m\approx\log(k)$
will be sufficient  to guarantee stability and 
quasi-optimality properties,
provided that $k^\alpha h\lesssim 1$ where $\alpha$ depends on the
stability and regularity of the continuous problem.
This constant $\alpha$ was the subject of Section \ref{DONALD_BROWN:stability}.
The conditions on $h$ are the same as for the standard $Q_1$ FEM on the
global fine scale. For example, in 2 dimensions, in the case of pure Robin boundary conditions on a convex domain, it is  required that $k^{3/2}h\lesssim 1$ for stability \cite{DONALD_BROWN:Wu2014CIPFEM}
and $k^2h\lesssim 1$ for quasi-optimality \cite{DONALD_BROWN:melenk_phd} 
is satisfied.

\section{Error Analysis}\label{DONALD_BROWN:s:erroranalysis}

The error analysis for the algorithm presented in Section \ref{DONALD_BROWN:s:method}, is very similar to that developed in \cite{DONALD_BROWN:Peterseim2014} and references therein, and in particular for the
Petrov-Galerkin formulation we discuss now in \cite{DONALD_BROWN:Gallistl.Peterseim:2015}. 
It is clear the proofs are  unaffected by the coefficients as the arguments rely on very general constants being bounded
such as $C_{a}$, $C_{stab}(k,\dmn,A,V^2)$, and $\gamma(k,\dmn,A,V^2)$, for example. 
This is primarily due to the upper and lower boundedness on the coefficients \eqref{DONALD_BROWN:upperbounds}.
However, we will highlight the main themes of the analysis here as this will be useful to refer to in our
discussion on Numerical Examples in Section \ref{DONALD_BROWN:numericalexamples} as well as general completeness of the discussion. 

We begin the error analysis with some notation. 
We denote the global finite element space on the fine scale by
$V_h:=V_h(\Omega)=\cS^1(\G_h)\cap V$.
We denote the solution operator of the truncated element corrector problem 
\eqref{DONALD_BROWN:e:lambdacorrectorproblem} by $\Cor_{T,m}$. 
Then,  any $z\in\N_H$ and any $T\in\G_H$ satisfy
$\lambda_{z,T} = \Cor_{T,m}(\Lambda_z)$ and we refer to
$\Cor_{T,m}$ as the truncated element correction operator.
The map $\Lambda_z\mapsto \lambda_z$
described in Subsection~\ref{DONALD_BROWN:ss:defMethod} defines a linear operator
$\Cor_m$ via $\Cor_m(\Lambda_z)=\lambda_z$ for any $z\in\N_H$,
referred to as correction operator.

For the analysis we introduce idealized counterparts of these
correction operators where the patch $\Omega_T$ equals $\Omega$. These global corrections are never computed and are merely used in the analysis. 
We define the null space 
$$
  W_h := \{ v\in V_h\,:\, I_H(v) = 0\},
$$
also referred to as the fine-scale space on the global domain.
For any $v\in V$, the idealized element corrector problem seeks
$\Cor_{T} v\in W_h$ such that
\begin{equation}\label{DONALD_BROWN:e:idealElementCorrProb}
a(w,\Cor_{T} v) = a_T(w,v)\quad\text{for all }w\in W_h.
\end{equation}
Furthermore, define
\begin{equation}\label{DONALD_BROWN:e:idealCorrector}
\Cor v:=\sum_{T\in\G_H} \Cor_{T} v.
\end{equation}

Recall, we proved in Section \ref{DONALD_BROWN:stability} that %
the form $a$ with heterogeneous coefficients given by \eqref{DONALD_BROWN:varform},
is continuous and there is a
constant $C_a$ such that
\begin{equation*}
a(v,w) \leq C_a \|v\|_V \|w\|_V
\quad\text{for all } v,w\in V.
\end{equation*}
The following result implies the well-posedness of the 
idealized corrector problems.

\begin{lemma}[Well-posedness for idealized corrector problems]
\label{DONALD_BROWN:l:wellposedideal}
Provided 
\begin{equation}\label{DONALD_BROWN:e:resolution}
\CIH\sqrt{\Col} H k  \leq 1,
\end{equation}
we have for all $w\in W_h$ equivalence of norms
\begin{equation*} 
A^{\frac{1}{2}}_{min}\|\nabla w\|_{L^2(\Omega)} 
\leq \|w\|_V 
\leq\left(V^2_{max}+A_{max}\right)^{\frac{1}{2}}\, \|\nabla w\|_{L^2(\Omega)},
\end{equation*}
and
coercivity
\begin{equation*} 
\left(V^2_{max}+A_{max}\right) \|\nabla w\|_{L^2(\Omega)}^2 \leq \operatorname{Re} a(w,w) .
\end{equation*}
\end{lemma}
\begin{proof} The lower bound is trivial, indeed we have that 
\begin{align*}
 \|w\|^2_V =\|k V w\|_{L^2(\Omega)}^2 +\|A^{\frac{1}{2}}\nabla w\|_{L^2(\Omega)}^2 \geq A^{}_{min}\|\nabla w\|_{L^2(\Omega)}^2. 
\end{align*}
For the upper bound, we note for  any $w\in W_h$ the property \eqref{DONALD_BROWN:e:IHapproxstab} implies
\begin{equation*}
k^2 \|V w\|_{L^2(\Omega)}^2 
= k^2 \|V (1-I_H)w\|_{L^2(\Omega)}^2
\leq
V^2_{max}\CIH^2\Col H^2 k^2 \|\nabla w\|_{L^2(\Omega)}^2.
\end{equation*}
Thus, using \eqref{DONALD_BROWN:e:resolution} we arrive at 
\begin{align*}
 \|w\|^2_V &=\|k V w\|_{L^2(\Omega)}^2 +\|A^{\frac{1}{2}}\nabla w\|_{L^2(\Omega)}^2\\
 &\leq V^2_{max}\CIH^2\Col H^2 k^2 \|\nabla w\|_{L^2(\Omega)}^2+ A_{max} \|\nabla w\|_{L^2(\Omega)}^2\\
 &\leq  \left(V^2_{max}+A_{max}\right)\|\nabla w\|_{L^2(\Omega)}^2.
\end{align*}
Note from this we have 
\begin{align*}
 \|k V w\|_{L^2(\Omega)}^2& \leq \left(V^2_{max}+A_{max}\right)\|\nabla w\|_{L^2(\Omega)}^2-\|A^{\frac{1}{2}}\nabla w\|_{L^2(\Omega)}^2\\
 & \leq \left(V^2_{max}+A_{max}-A_{min}\right)\|\nabla w\|_{L^2(\Omega)}^2,
\end{align*}
and so 
\begin{align*}
 \operatorname{Re} a(w,w)&=\|A^{\frac{1}{2}}\nabla w\|_{L^2(\Omega)}^2-\|k V w\|_{L^2(\Omega)}^2\\
 &\geq  \left(V^2_{max}+A_{max}\right)\|\nabla w\|_{L^2(\Omega)}^2.
\end{align*}
Thus,  equivalence and 
coercivity
is proven.
\qed
\end{proof}

Lemma~\ref{DONALD_BROWN:l:wellposedideal} implies that the idealized corrector
problems \eqref{DONALD_BROWN:e:idealCorrector}
are well-posed and the correction operator $\Cor$
is continuous in the sense that 
\begin{equation*}
\|\Cor v_H \|_V \leq C_\Cor \|v_H\|_V
\quad\text{for all } v_H \in V_H
\end{equation*}
for some constant $C_\Cor \approx 1$.
Since the inclusion $W_h(\Omega_T)\subseteq W_h$ holds,
the well-posedness result of Lemma~\ref{DONALD_BROWN:l:wellposedideal} carries over
to the corrector problems \eqref{DONALD_BROWN:e:lambdacorrectorproblem}
in the subspace $W_h(\Omega_T)$
with the sesquilinear form $a_{\Omega_T}$.

Again as with the homogeneous coefficient case \cite{DONALD_BROWN:Gallistl.Peterseim:2015}, the proof of well-posedness of the Petrov-Galerkin method
\eqref{DONALD_BROWN:e:discreteproblem} is 
based on the fact that the difference $(\Cor-\Cor_m)(v)$
decays exponentially with the distance from $\supp(v)$.
In the next theorem, we quantify the difference between the
idealized and the discrete correctors.
As the proof is a bit technical and does not differ fundamentally from the homogeneous case, 
we refer the reader to  Appendix of  \cite{DONALD_BROWN:Gallistl.Peterseim:2015} and references therein.
The proof is based on the exponential decay of the corrector $\Cor \Lambda_z$
and  requires the resolution
condition \eqref{DONALD_BROWN:e:resolution}, namely $k H \lesssim 1$.

\begin{theorem}\label{DONALD_BROWN:t:CorrCloseness}
Under the resolution condition \eqref{DONALD_BROWN:e:resolution},
there exist constants 
$C_1\approx 1 \approx C_2$ and $0<\theta<1$ such that
any $v\in V_H$, any $T\in\G_H$ and any $m\in\mathbb N$ satisfy
\begin{align}
\label{DONALD_BROWN:e:CorrCloseness1}
\|\nabla(\Cor_{T} v - \Cor_{T,m} v)\|_{L^2(\Omega)}
&
\leq C_1
 \theta^m \|\nabla v \|_{L^2(T)},
\\
\label{DONALD_BROWN:e:CorrCloseness2}
\|\nabla(\Cor v - \Cor_m v)\|_{L^2(\Omega)}
&
\leq C_2 \sqrt{\Colm}
\theta^m \|\nabla v \|_{L^2(\Omega)}.
\end{align}
\end{theorem}
\begin{proof}
 See Appendix of  \cite{DONALD_BROWN:Gallistl.Peterseim:2015}. \qed
\end{proof}

Provided we choose the fine-mesh $h$ small enough,
the standard finite element over the mesh $\G_h$ is stable in the sense that
there exists a constant $C_{\mathrm{FEM}}$
such that with $\gamma(k,\Omega,A,V^2)$ 
from \eqref{DONALD_BROWN:e:helmholtzstability} 
it holds that
\begin{equation}\label{DONALD_BROWN:e:finehstability}
 \big(C_{\mathrm{FEM}}\gamma(k,\Omega,A,V^2) \big)^{-1}
  \leq
  \inf_{v\in V_h\setminus\{0\}}\sup_{w\in V_h\setminus\{0\}} 
      \frac{\operatorname{Re} a(v,w)}{\|v\|_V\|w\|_V} .
\end{equation}
Recall, this is actually a condition on the fine-scale parameter $h$.
In general, the requirements on $h$ depend on the stability of the 
continuous problem \cite{DONALD_BROWN:melenk_phd}. We now recall the conditions on the oversampling parameter for the well-posedness of the discrete problem. Again, the proof here
does not rely heavily on the coefficients, just the general boundedness and ellipticity constants etc. Thus, we again refer the reader to \cite{DONALD_BROWN:Gallistl.Peterseim:2015}.

\begin{theorem}[Well-posedness of the discrete problem]\label{DONALD_BROWN:t:wellposed}
Under the resolution conditions \eqref{DONALD_BROWN:e:resolution}
and \eqref{DONALD_BROWN:e:finehstability}
and the following oversampling condition
\begin{equation}\label{DONALD_BROWN:e:mcondition}
m\gtrsim
\lvert\log\big( C_{\mathrm{FEM}}\gamma(k,\Omega,A,V^2)\big)\rvert
\Big/\lvert \log(\theta)\rvert,
\end{equation}
problem \eqref{DONALD_BROWN:e:discreteproblem} is well-posed and the constant
$C_{\mathrm{PG}}:=2C_{I_H,V}C_\Cor C_{\mathrm{FEM}}$ satisfies
\begin{equation*}
\big(C_{\mathrm{PG}}\gamma(k,\Omega,A,V^2)\big)^{-1}
\leq
 \inf_{v_H\in V_H\setminus\{0\}}
 \sup_{\tilde v_H\in \widetilde V_H\setminus\{0\}}
 \frac{\operatorname{Re} a(v_H,\tilde v_H)}{\|v_H\|_V\|\tilde v_H\|_V} .
\end{equation*}
\end{theorem}
\begin{proof}
 See   \cite{DONALD_BROWN:Gallistl.Peterseim:2015}. \qed
\end{proof}

The quasi-optimality  requires the following additional condition
on the oversampling parameter $m$,
\begin{equation}\label{DONALD_BROWN:e:oversampling2}
m\gtrsim
\lvert
\log\Big( 
 C_{\mathrm{PG}}\gamma(k,\Omega,A,V^2) 
\Big)
\rvert
\Big/ \lvert \log(\theta) \rvert.
\end{equation}

\begin{theorem}[Quasi-optimality]\label{DONALD_BROWN:t:quasiopt}
 The resolution conditions \eqref{DONALD_BROWN:e:resolution} and
 \eqref{DONALD_BROWN:e:finehstability}
 and the oversampling
 conditions \eqref{DONALD_BROWN:e:mcondition} and 
 \eqref{DONALD_BROWN:e:oversampling2}
 imply that the solution $u_H$ to
 \eqref{DONALD_BROWN:e:discreteproblem} with parameters $H$, $h$, and $m$ 
 and the solution $u_h$ of the standard Galerkin FEM on the mesh
 $\G_h$ satisfy
\begin{equation*}
\|u_h - u_H\|_V
\lesssim
  \|(1-I_H) u_h \|_V
\approx 
  \min_{v_H\in V_H} \| u_h - v_H \|_V .
\end{equation*}
\end{theorem}
\begin{proof}
 See   \cite{DONALD_BROWN:Gallistl.Peterseim:2015}. \qed
\end{proof}

The following consequence of Theorem~\ref{DONALD_BROWN:t:quasiopt} states an
estimate for the error $u-u_H$.

\begin{corollary}
Under the conditions of Theorem~\ref{DONALD_BROWN:t:quasiopt}, the discrete
solution $u_H$ to \eqref{DONALD_BROWN:e:discreteproblem} satisfies with some
constant $C\approx 1$ that
\begin{equation*}
\| u- u_H \|_V
\leq
  \| u- u_h \|_V
  +
  C \min_{v_H\in V_H} \| u_h - v_H \|_V .
  \qquad
\end{equation*}

\end{corollary}

For the class of coefficients described in 
Theorem~\ref{DONALD_BROWN:Theorem1}, this leads to the following
convergence rates. Provided that the geometry allows for $H^2$
regularity of the solution and that $h$ is sufficiently small
such that the standard FEM is quasi-optimal on the fine scale
$h$ and the error is dominated by the coarse-scale part,
we have
$$
\| u- u_H \|_V \leq \mathcal{O}(kH) .
$$

\section{Numerical Examples}\label{DONALD_BROWN:numericalexamples}

In this section, we present the results from our numerical experiments
on a smooth coefficient for 
 both cases when the conditions are satisfied and when it is violated. 
 Further, we implement the method on discontinuous periodic coefficients to highlight  broader applicability of the method.
We give 3 example coefficients; 
based on \eqref{DONALD_BROWN:numcoef1}, \eqref{DONALD_BROWN:numcoef2}, and a discontinuous example.
In all three experiments we took $\Omega=(-1,1)^2$ to be the unit square.
We use triangular meshes and continuous $P_1$ finite elements as
trial functions.
We used $k=2^5$, $g=0,$ and the approximate  point source
$$
f(x) =
\begin{cases}
\exp\left(-\frac{1}{1-(20|x|)^2}\right)& \text{for } |x| < 1/20\\
0 & \text{else}.
\end{cases}
$$
The coarse-scale mesh-sizes are $H=2^{-3},2^{-4},2^{-5},2^{-6}$
and the fine-scale mesh-size is $h=2^{-8}$.

The convergence history plots display the errors in the 
$\|\cdot\|_V$ norm as well as $L^2$ norms.
We compare the multiscale Petrov-Galerkin method for
oversampling parameters $m=1,2,3$ with the standard $P_1$
finite element method and the best-approximation.
To compute the error quantity we take the standard finite
element solution at the fine scale $h$ to be the overkill solution.

\begin{figure}[pt]
\subfigure[The coefficient $V^2$ for example 1. \label{DONALD_BROWN_subfigCoeff1}]{
\includegraphics*[width=.45\textwidth]{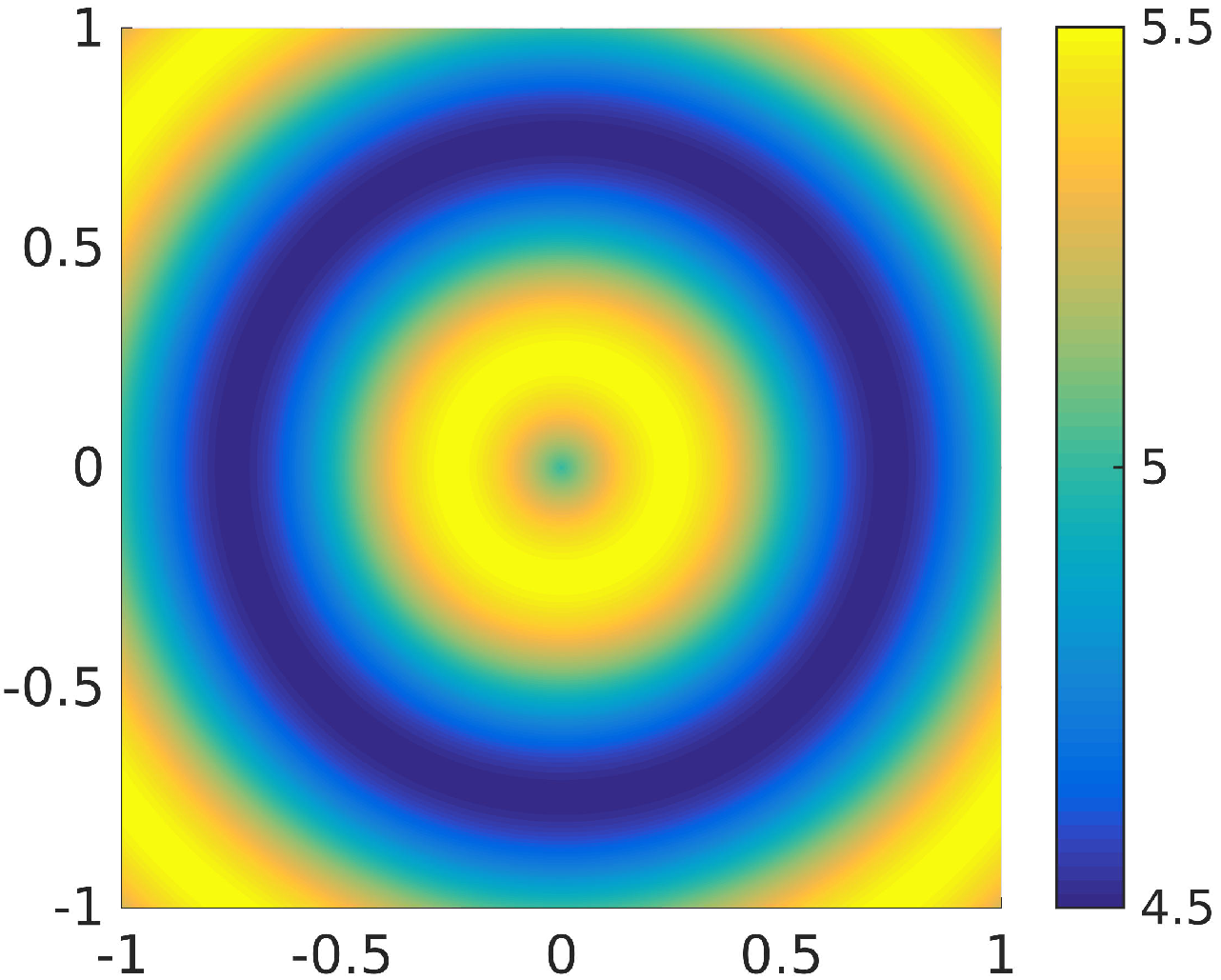}
}
\subfigure[Plot of the solution for example 1. \label{DONALD_BROWN_subfigSol1}]{
\includegraphics*[width=.45\textwidth]{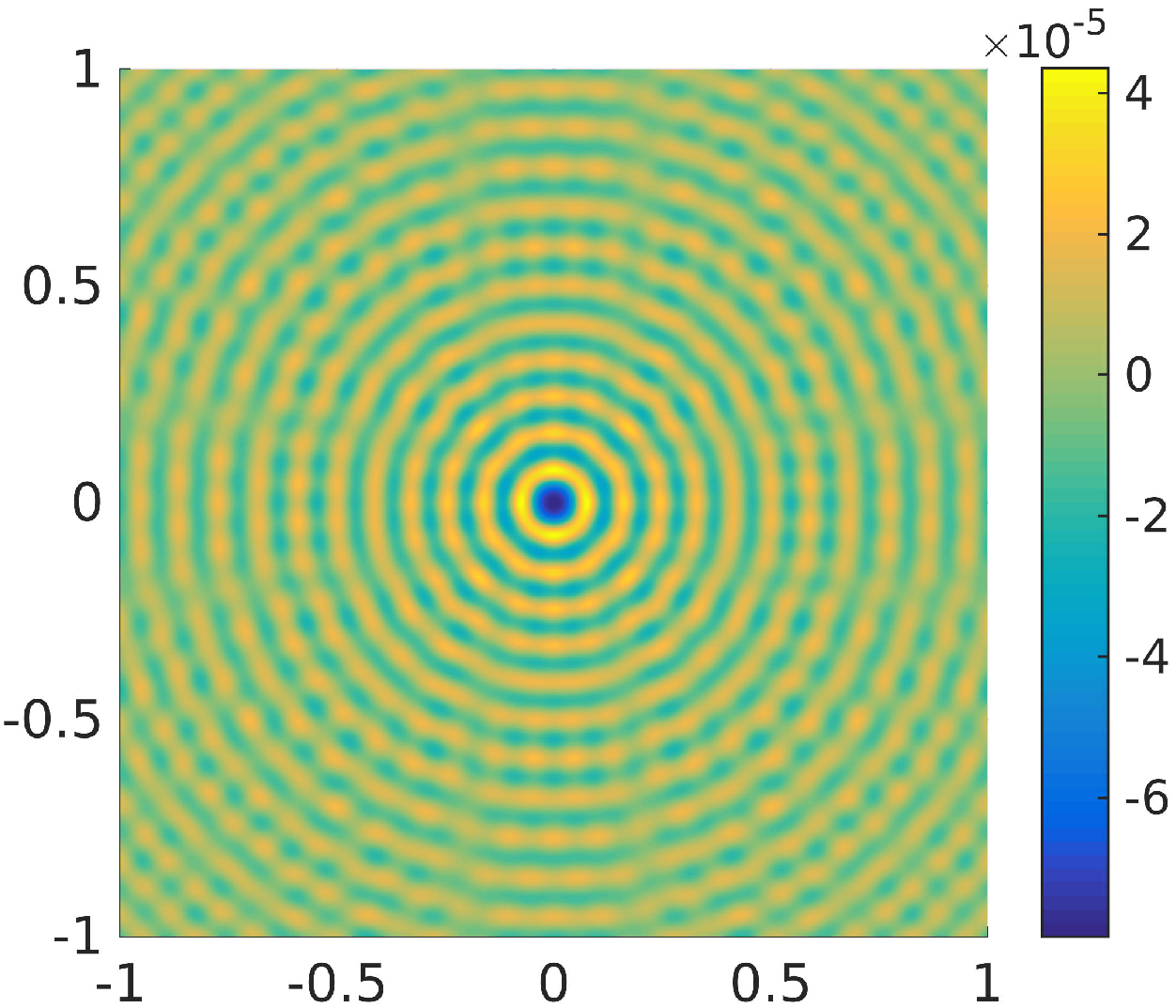}
}
\caption{Plots for example 1.\label{DONALD_BROWN_figVisualize1}}
\end{figure}

For the first example, we take $A=1$ and  $V^2$ as \eqref{DONALD_BROWN:numcoef1}. with $\epsilon = 1$ and refer to this as example 1.
Note that this does not violate the stability condition.
The coefficient $V^2$ is displayed in Figure~\ref{DONALD_BROWN_subfigCoeff1} and the corresponding computational solution is displayed in Figure~\ref{DONALD_BROWN_subfigSol1}.
Figures~\ref{DONALD_BROWN_subfigConvhistV1}--\ref{DONALD_BROWN_subfigConvhistL21}
display the convergence history in the $V$-norm and the $L^2$ norm 
for example 1. In general, we see that the multiscale method 
appears to perform much better than the corresponding standard $P_1$ finite element. However, there appears to be some
resonance effects of some sort that is particularly pronounced in
the $V$ norm just before the resolution condition is satisfied.
This is not in contradiction with the theory. It has been 
demonstrated in \cite{DONALD_BROWN:Gallistl.Peterseim:2015} that there
is no decay of the corrector functions if the resolution condition
is not satisfied, so that in this regime the localization is not
justified and leads to unreliable results.

\begin{figure}[pt]
\subfigure[Convergence in $V$ norm: example 1. \label{DONALD_BROWN_subfigConvhistV1}]{
\includegraphics*[width=.45\textwidth]{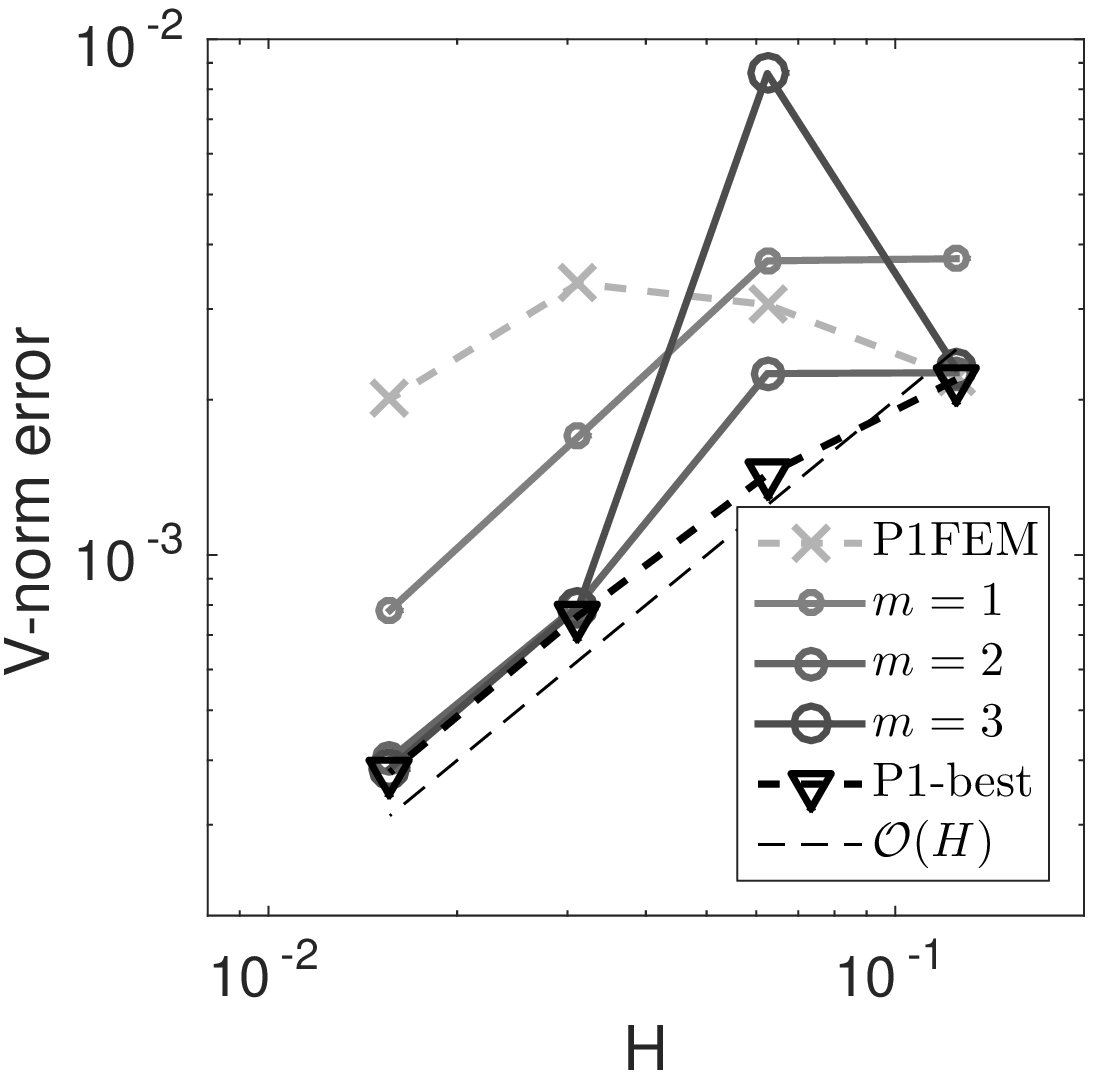}
}
\subfigure[Convergence in $L^2$ norm:  example 1. \label{DONALD_BROWN_subfigConvhistL21}]{
\includegraphics*[width=.45\textwidth]{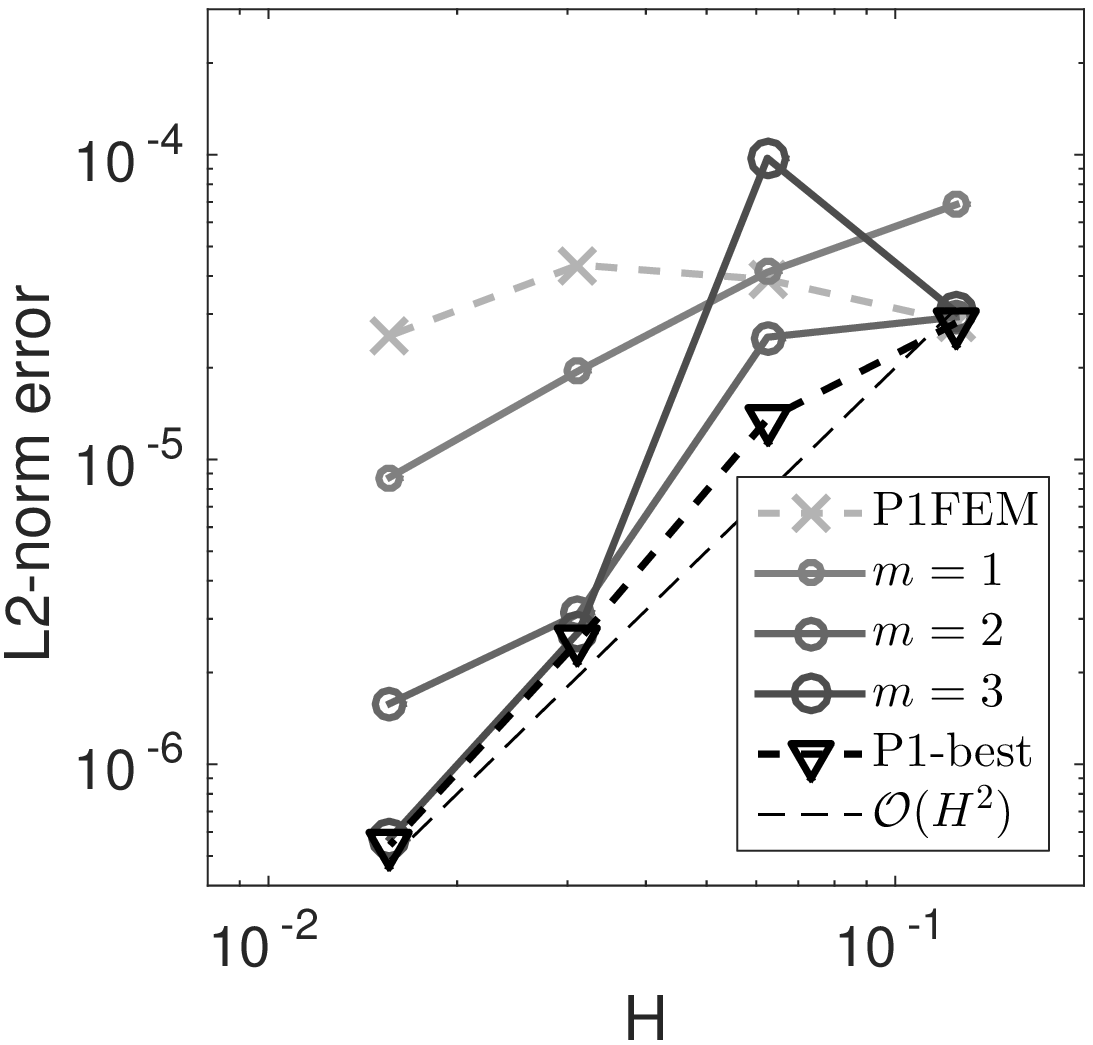}
}
\caption{Convergence history for example 1.\label{DONALD_BROWN_figConvergence1}}
\end{figure}

For the second example, we take $A=V^2$ and  $V^2$ as \eqref{DONALD_BROWN:numcoef2}, and refer to this as example 2.
For the parameters we took $\delta=1$, $\epsilon = 0.1$, $\alpha=0.08$, and note that the corresponding stability condition
$\alpha\exp(2\alpha)< \epsilon$ is narrowly satisfied. 
\begin{figure}[pt]
\subfigure[The coefficient $V^2$ for example 2. \label{DONALD_BROWN_subfigCoeff2}]{
\includegraphics*[width=.45\textwidth]{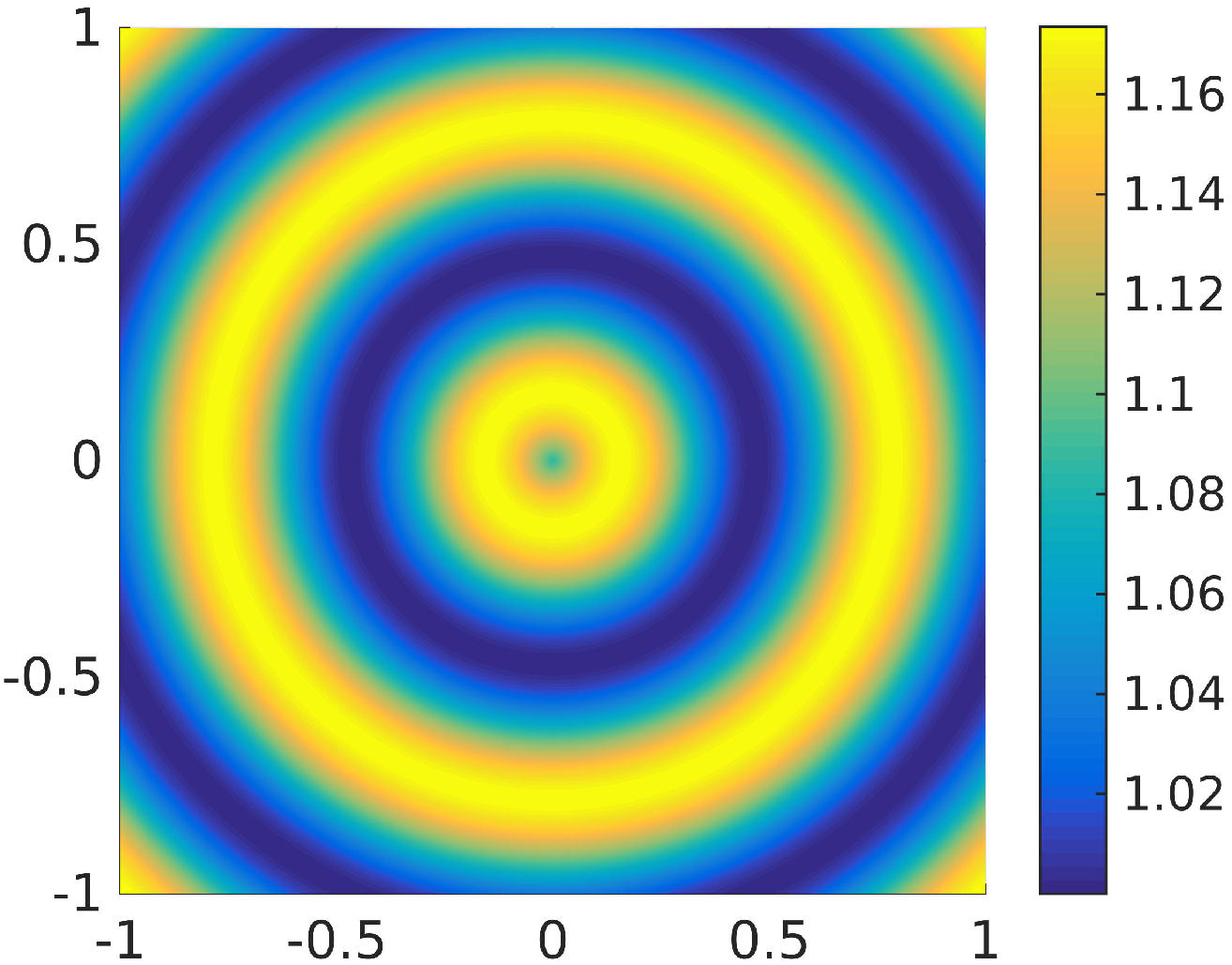}
}
\subfigure[Plot of the solution for example 2. \label{DONALD_BROWN_subfigSol2}]{
\includegraphics*[width=.45\textwidth]{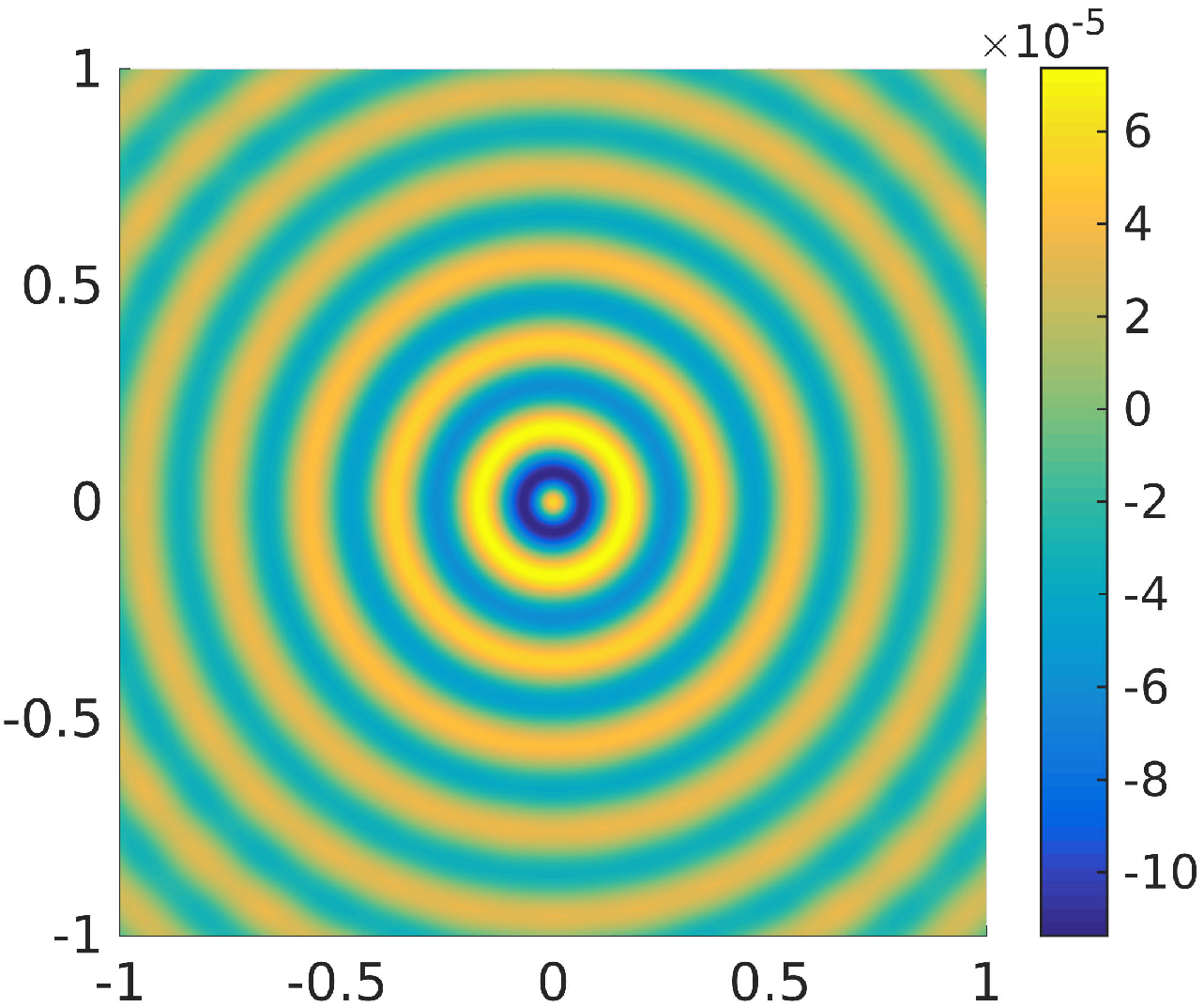}
}
\caption{Plots for example 2.\label{DONALD_BROWN_figVisualize2}}
\end{figure}
The coefficient $V^2$ is displayed in Figure~\ref{DONALD_BROWN_subfigCoeff2} and the computational solution is displayed in Figure~\ref{DONALD_BROWN_subfigSol2}.
Figures~\ref{DONALD_BROWN_subfigConvhistV2}--\ref{DONALD_BROWN_subfigConvhistL22}
display the convergence history in the $V$-norm and the $L^2$ norm 
for example 2. We see that in this example, we achieve faster convergence and do not see the resonance effects. This is also the case for the standard finite elements. 
\begin{figure}[pt]
\subfigure[Convergence  in $V$ norm example 2. \label{DONALD_BROWN_subfigConvhistV2}]{
\includegraphics*[width=.45\textwidth]{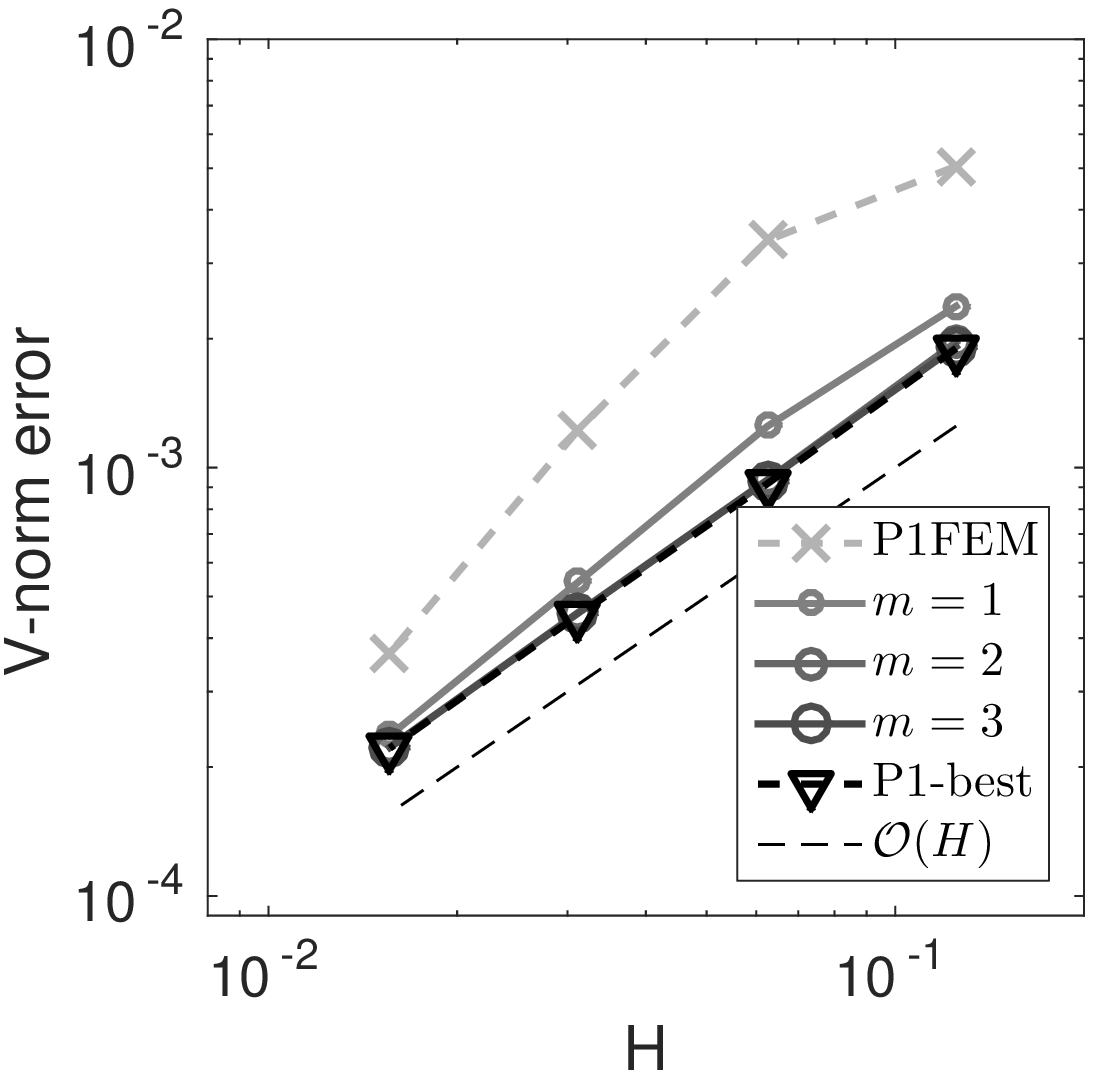}
}
\subfigure[Convergence  in $L^2$ norm  example 2. \label{DONALD_BROWN_subfigConvhistL22}]{
\includegraphics*[width=.45\textwidth]{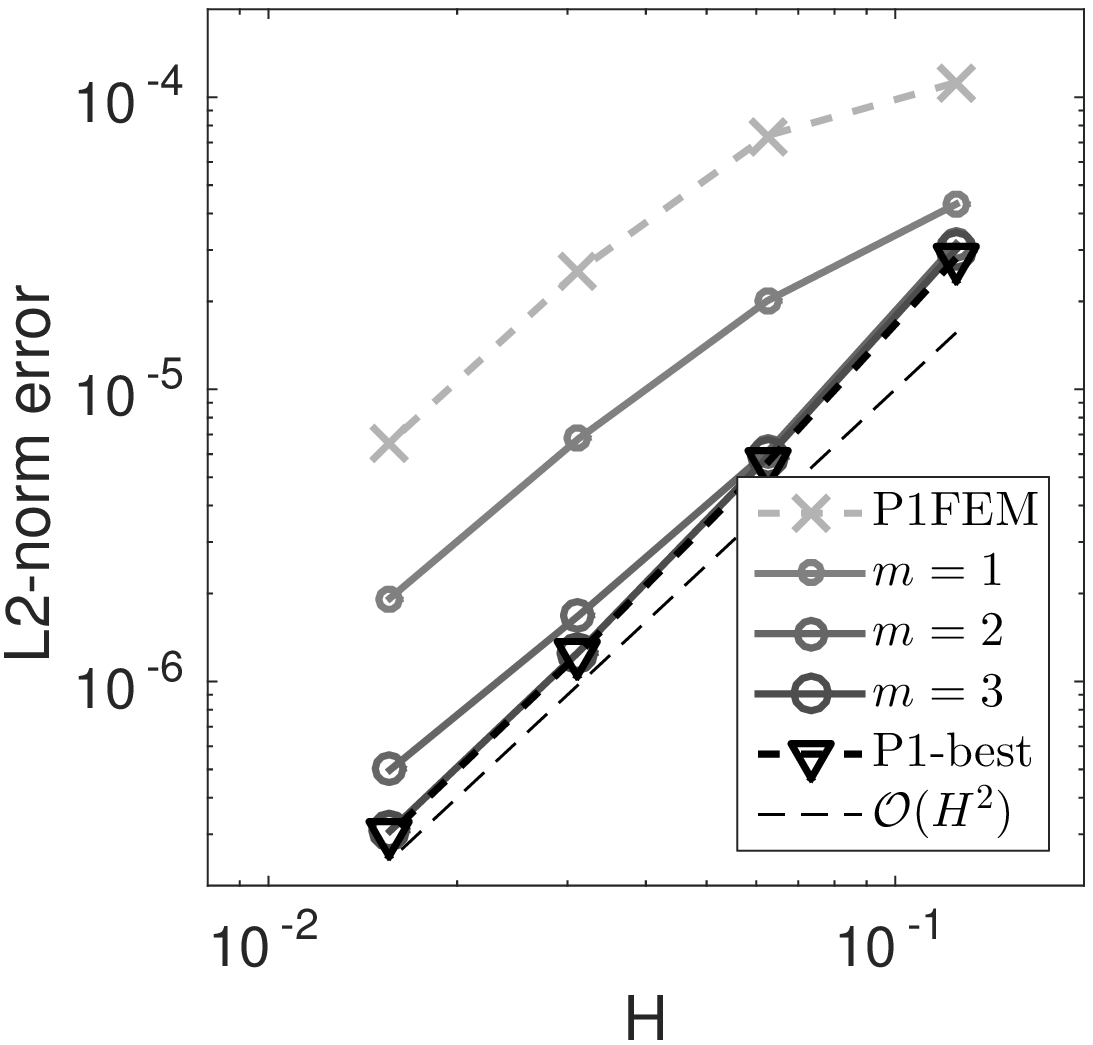}
}
\caption{Convergence history for example 2.\label{DONALD_BROWN_figConvergence2}}
\end{figure}


\begin{figure}[pt]
\subfigure[The  coefficient $V^2$ for example 3. \label{DONALD_BROWN_subfigCoeff3}]{
\includegraphics*[width=.45\textwidth]{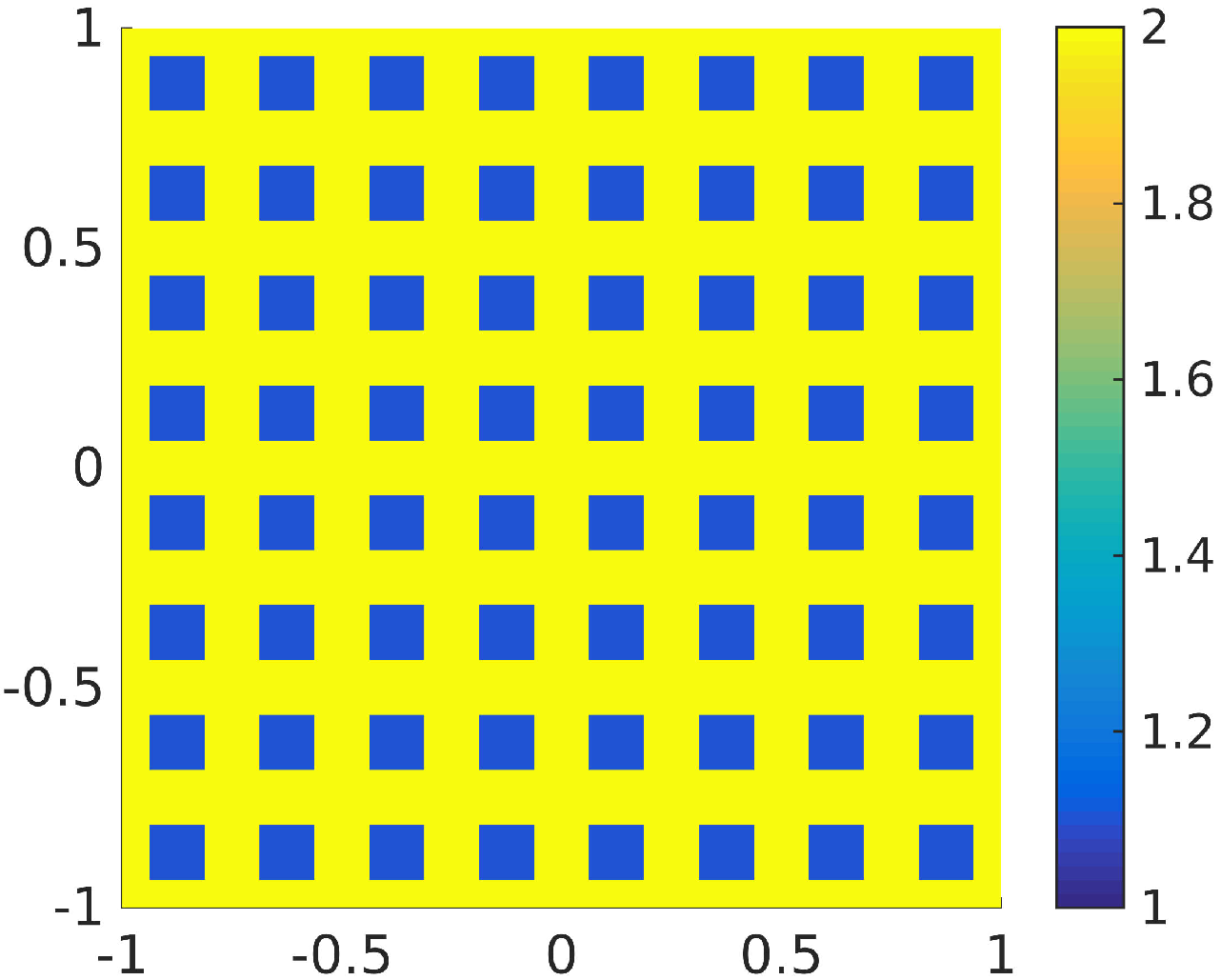}
}
\subfigure[Plot of the solution for example 3. \label{DONALD_BROWN_subfigSol3}]{
\includegraphics*[width=.45\textwidth]{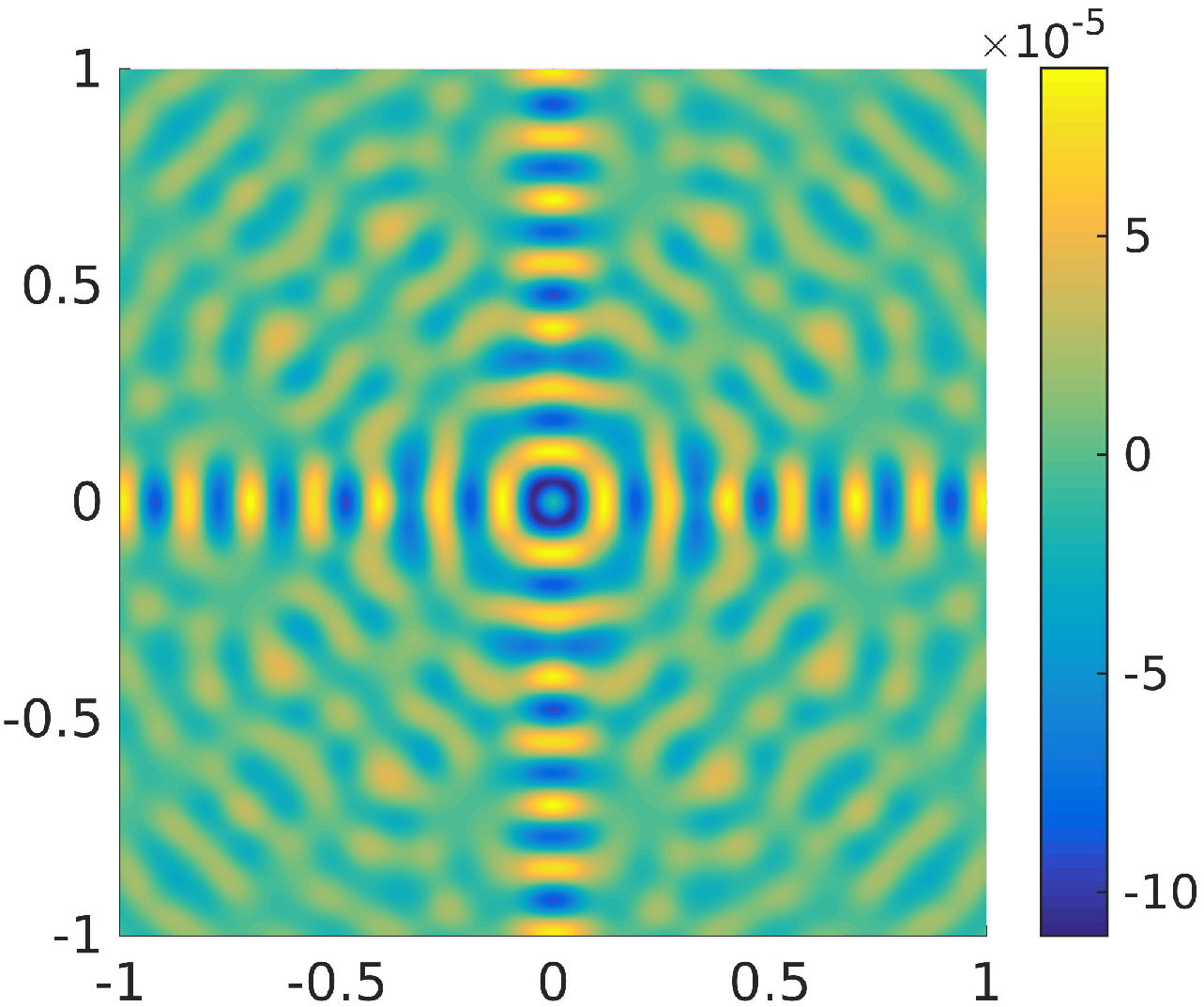}
}
\caption{Plots for example 3.\label{DONALD_BROWN_figVisualize3}}
\end{figure}

We now present a numerical example outside of our stability theory.
We take $V^2=2$ except at periodically placed blocks where $V^2=1$ 
and plot the function in  Figure~\ref{DONALD_BROWN_subfigCoeff3}. 
We refer to this as example 3.
The computational solution is displayed in Figure~\ref{DONALD_BROWN_subfigSol3}.
\begin{figure}[pt]
\subfigure[Convergence in $V$ norm  example 3. \label{DONALD_BROWN_subfigConvhistV3}]{
\includegraphics*[width=.45\textwidth]{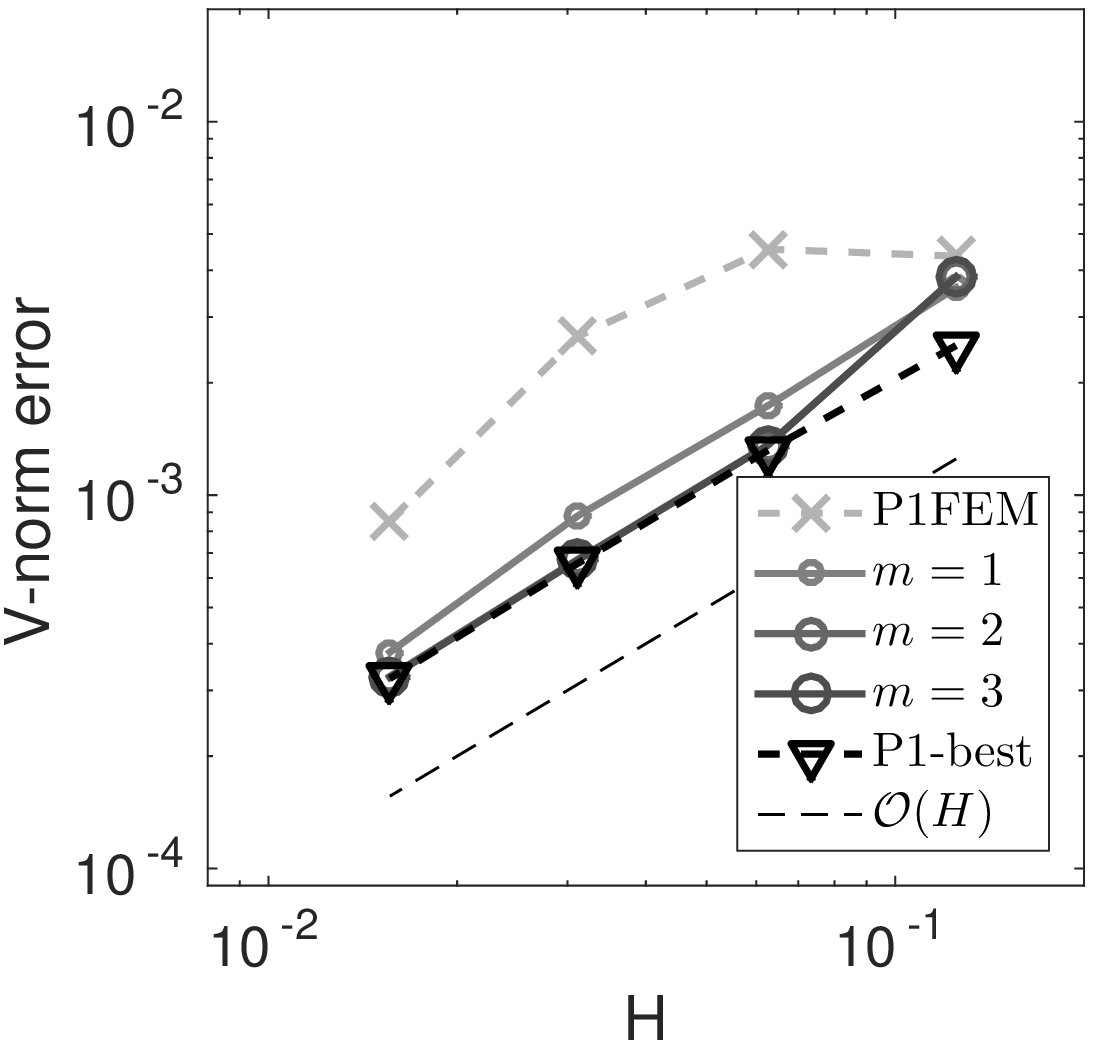}
}
\subfigure[Convergence  in $L^2$ norm   example 3. \label{DONALD_BROWN_subfigConvhistL23}]{
\includegraphics*[width=.45\textwidth]{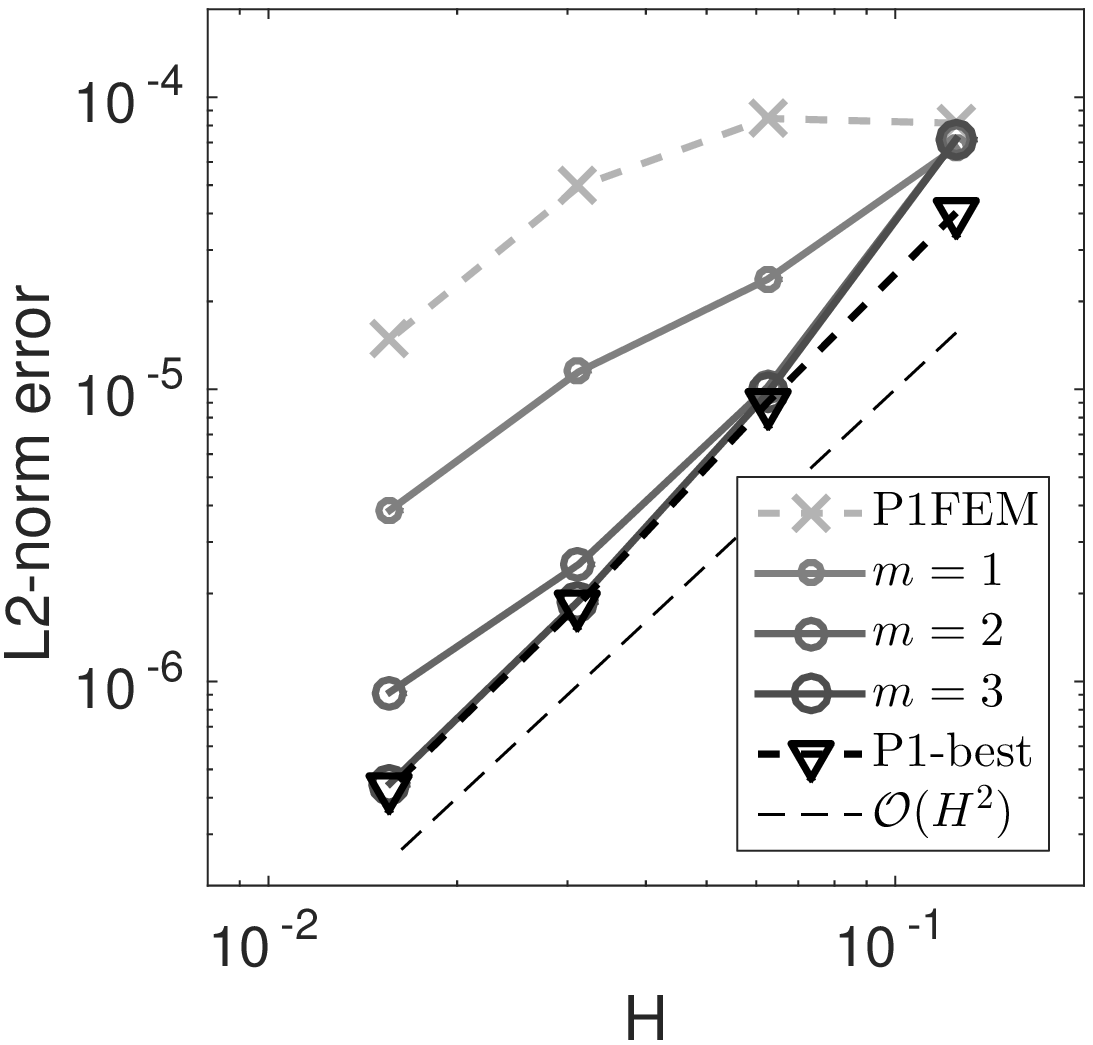}
}
\caption{Convergence history for example 3.\label{DONALD_BROWN_figConvergence3}}
\end{figure}
Figures~\ref{DONALD_BROWN_subfigConvhistV3}--\ref{DONALD_BROWN_subfigConvhistL23}
display the convergence history in the $V$-norm and the $L^2$ norm 
for example 3. We observe that the method performs particularly well in this example, especially when compared against the corresponding $P_1$ finite element. 
We do not see the resonances as with example 1.

\section{Conclusions}

In this work, we developed a multiscale method to efficiently solve the heterogeneous Helmholtz equation at high frequency.
The primary challenge was establishing $k$-explicit bounds for the continuous problem as these are critical in the analysis of the patch truncation parameter.
We established these bounds for a class of smooth coefficients given some restrictions that appear to depend heavily on the frequency of oscillations and the amplitude of the
coefficients. We then presented our multiscale method whose error analysis is not significantly modified by the heterogeneities assuming standard upper and lower boundedness.
Finally, we implemented the algorithm on two coefficients that fit inside the class of coefficients in our main theorem and one that is discontinuous. We see that the
method performs well in these cases. Future work includes exploring if these stability estimates apply to a greater class of more heterogeneous coefficients with less smoothness.

\section*{Appendix: Proof of Stability}\label{DONALD_BROWN:proofappendix}
\addcontentsline{toc}{section}{Appendix: Proof of Stability}

\subsection*{Technical and Auxiliary Lemmas }
We will now proceed by recalling and demonstrating a few technical and auxiliary Lemmas used in the proof of Theorem \ref{DONALD_BROWN:Theorem1}. 
We begin with two critical technical lemmas that remain unchanged from the homogeneous case examined in \cite{DONALD_BROWN:hetmaniuk} and are repeated here for completeness. 
\begin{lemma}
 Let $m\in W^{1,\infty}(\dmn)^d$ and for  all $q\in H^1(\dmn)$ we have 
 \end{lemma}
 \begin{align}\label{DONALD_BROWN:id1}
\int_{\partial \dmn} |q|^2m \cdot \nu ds =\int_{\dmn} \mbox{div}(m)|q|^2dx+ 2 \operatorname{Re}\int_{\dmn}q m \cdot \nabla \bar{q}dx.
 \end{align}
\begin{proof}
 See \cite{DONALD_BROWN:hetmaniuk}, Lemma 3.1.\qed
\end{proof}
\begin{lemma}
 Let $m\in W^{1,\infty}(\dmn)^d$ and for  all $q\in H^1_{\dir}(\dmn)\cap H^{3/2+\delta}, \delta>0,$ we have 
 \end{lemma}
 \begin{align}
&\int_{\partial \dmn\backslash \dir} |\nabla q|^2 m \cdot \nu ds -\int_{ \dir} |\partial_{\nu}q|^2 m \cdot \nu ds\nonumber\\
&=\int_{\dmn} \mbox{div}(m)|\nabla q|^2dx- 2 \operatorname{Re}\int_{\dmn}\nabla q  \cdot( \nabla \bar{q}\nabla) m dx\nonumber \\
\label{DONALD_BROWN:id2}
&- 2 \operatorname{Re}\int_{\dmn}\Delta  q  (m \cdot \nabla \bar{q}) dx+ 2 \operatorname{Re}\int_{\partial \dmn\backslash \dir}\partial_{\nu}  q  (m \cdot \nabla \bar{q}) ds
\end{align}
\begin{proof}
 See  \cite{DONALD_BROWN:Grisvard}. \qed
\end{proof}

Here we will present a few auxiliary Lemmas.

\begin{lemma} Let $\dmn \subset \mathbb{R}^d$ be a bounded connected Lipschitz domain. Let $u\in H^1(\dmn)$ be a weak solution of \eqref{DONALD_BROWN:main}, with $f\in L^2(\dmn)$ and $g\in L^2( \rob)$. Then, we have for 
any
$\epsilon>0$
\begin{align}\label{DONALD_BROWN:Lemmaest1}
{k^2}  \TwoNorm{u}{\rob}^2 & \leq \frac{1}{\beta_{min}}\left(\frac{1}{\epsilon}\TwoNorm{f}{\dmn}^2 +{k^2\epsilon} \TwoNorm{u}{\dmn}^2+\frac{1}{\beta_{min}}\TwoNorm{g}{\rob}^2\right).
\end{align}
\end{lemma}
\begin{proof}
Taking $v=u$ into the variational form \eqref{DONALD_BROWN:var}  and looking at the imaginary part we have
\begin{align*}
{\Im}(a(u,u))=-(k \beta(x) u, u)={\Im}((g,u)_{L^2(\rob)}+(f,u)_{L^2(\dmn)}),
\end{align*}
and so 
\begin{align*}
&
k \beta_{min} \TwoNorm{u}{\rob}^2
\\
&\quad \leq \TwoNorm{u}{\dmn}  \TwoNorm{f}{\dmn}+\TwoNorm{u}{\rob}  \TwoNorm{g}{\rob}
\\
&\quad \leq \frac{1}{2k\xi_{1}}\TwoNorm{f}{\dmn}^2 +\frac{k\xi_{1}}{2} \TwoNorm{u}{\dmn}^2+\frac{1}{2\xi_2}\TwoNorm{g}{\rob}^2+ \frac{\xi_{2}}{2} \TwoNorm{u}{\rob}^2.
\end{align*}
Multiplying by $k$, 
dividing by $\beta_{min}$,
and setting $\xi_{2}=\beta_{min}k$ we obtain 
\begin{align*}
k^2  \TwoNorm{u}{\rob}^2  
\leq 
 \frac{1}{\beta_{min}}\bigg(\frac{1}{2\xi_{1}}\TwoNorm{f}{\dmn}^2
 &+\frac{k^2\xi_1}{2} \TwoNorm{u}{\dmn}^2
\\
 &+\frac{1}{2\beta_{min}}\TwoNorm{g}{\rob}^2+ \frac{k^2\beta_{min}}{2} \TwoNorm{u}{\rob}^2\bigg),
\end{align*}
and we obtain 
\begin{align*}
\frac{k^2}{2}  \TwoNorm{u}{\rob}^2 & \leq \frac{1}{\beta_{min}}\left(\frac{1}{2\xi_{1}}\TwoNorm{f}{\dmn}^2 +\frac{k^2\xi_{1}}{2} \TwoNorm{u}{\dmn}^2+\frac{1}{2\beta_{min}}\TwoNorm{g}{\rob}^2\right).
\end{align*}
Taking $\xi_{1}=\epsilon>0$ we arrive at the estimate. \qed

\end{proof}
We will also need the estimate below. 

\begin{lemma} Let $\dmn \subset \mathbb{R}^d$ be a bounded connected Lipschitz domain. Let $u\in H^1(\dmn)$ be a weak solution of \eqref{DONALD_BROWN:main} with $f\in L^2(\dmn)$ and $g\in L^2( \rob)$. Then, we have 
\begin{equation}\label{DONALD_BROWN:Lemmaest2}
\begin{aligned}
&\TwoNorm{\nabla u}{\dmn}^2 \\
&\quad\leq \frac{1}{A_{min}} \bigg[k^2\left( V_{max}^2+ \frac{\xi_{4}}{\beta_{min}}+\frac{ \xi_{3}}{2}\right)\TwoNorm{u}{\dmn}^2\\
&\qquad\qquad+\left(\frac{1}{2 k^2 \xi_{3}}+\frac{1}{\beta_{min}\xi_{4}} \right)\TwoNorm{f}{\dmn}^2+  \left(\frac{1}{\beta^2_{min}}+\frac{1}{4k^2}   \right)\TwoNorm{g}{\rob}^2\bigg].
\end{aligned}
\end{equation}
for 
any
$\xi_{3},\xi_{4}>0$.
\end{lemma}
\begin{proof}
Taking $v=u$ into the variational form \eqref{DONALD_BROWN:var}  and looking at the real part we have
\begin{align*}
{\operatorname{Re}}(a(u,u))&=(A(x) \nabla u,\nabla u)_{L^2(\dmn)}-(k^2 V^2(x) u ,u)_{L^2(\dmn)}\\
&={\operatorname{Re}}((g,u)_{L^2(\rob)}+(f,u)_{L^2(\dmn)}),
\end{align*}
and so we have 
\begin{align*}
\TwoNorm{A^{\frac{1}{2}}\nabla u}{\dmn}^2 \leq k^2\TwoNorm{V u }{\dmn}^2+\TwoNorm{u}{\dmn}  \TwoNorm{f}{\dmn}+\TwoNorm{u}{\rob}  \TwoNorm{g}{\rob}.
\end{align*}
Using the maximal and minimal values we have 
for any $\xi_3>0$ that

\begin{align}
A_{min}\TwoNorm{\nabla u}{\dmn}^2 &\leq k^2\TwoNorm{V u }{\dmn}^2+\TwoNorm{u}{\dmn}  \TwoNorm{f}{\dmn}+\TwoNorm{u}{\rob}  \TwoNorm{g}{\rob}\nonumber\\
\label{DONALD_BROWN:est1}
&\leq\left( k^2V_{max}^2+\frac{k^2 \xi_{3}}{2}\right)\TwoNorm{u}{\dmn}^2+ \frac{1}{2 k^2 \xi_{3}} \TwoNorm{f}{\dmn}^2 \nonumber\\
&+\frac{1}{4k^2}\TwoNorm{g}{\rob}^2+k^2\TwoNorm{ u}{\rob}^2.
\end{align}
Using estimate \eqref{DONALD_BROWN:Lemmaest1} we may write  
for any $\epsilon>0$

\begin{align}
{k^2}  \TwoNorm{u}{\rob}^2 & \leq \frac{1}{\beta_{min}}\left({k^2\epsilon} \TwoNorm{u}{\dmn}^2+\frac{1}{\epsilon}\TwoNorm{f}{\dmn}^2 +\frac{1}{\beta_{min}}\TwoNorm{g}{\rob}^2\right).
\end{align}

Inserting the above inequality into \eqref{DONALD_BROWN:est1} we obtain
\begin{align*}
&A_{min}\TwoNorm{\nabla u}{\dmn}^2 \\
&\qquad\leq\left( k^2V_{max}^2+\frac{k^2 \xi_{3}}{2}\right)\TwoNorm{u}{\dmn}^2+ \frac{1}{2 k^2 \xi_{3}} \TwoNorm{f}{\dmn}^2+  \frac{1}{4k^2}\TwoNorm{g}{\rob}^2\\
&\qquad+  \frac{1}{\beta_{min}}\left({k^2\epsilon} \TwoNorm{u}{\dmn}^2+\frac{1}{\epsilon}\TwoNorm{f}{\dmn}^2 +\frac{1}{\beta_{min}}\TwoNorm{g}{\rob}^2\right).
\end{align*}
Taking $\epsilon=\xi_{4}$ the above inequality becomes 
\begin{align*}
A_{min}\TwoNorm{\nabla u}{\dmn}^2 &\leq k^2\left( V_{max}^2+ \frac{\xi_{4}}{\beta_{min}}+\frac{ \xi_{3}}{2}\right)\TwoNorm{u}{\dmn}^2\\
&+ \left(\frac{1}{2 k^2 \xi_{3}}+\frac{1}{\beta_{min}\xi_{4}} \right)\TwoNorm{f}{\dmn}^2+  \left(\frac{1}{\beta^2_{min}}+\frac{1}{4k^2}   \right)\TwoNorm{g}{\rob}^2.
\end{align*}
Thus, we obtained our estimate. \qed
\end{proof}

\subsection*{Proof of Main Stability Result}

We are now in a position to prove Theorem \ref{DONALD_BROWN:Theorem1}. The key observation is that the Laplacian may be rewritten using \eqref{DONALD_BROWN:main} and combined
with the technical and auxiliary lemmas. This leads to the conditions on the coefficients \eqref{DONALD_BROWN:conditions}.

\begin{proof}[Proof of Theorem \ref{DONALD_BROWN:Theorem1}]

Using \eqref{DONALD_BROWN:id2} where we write $$-\Delta u=\frac{1}{A}(f+k^2V^2u+\nabla A \cdot \nabla u),$$ 
  $\partial_{\nu}u=0$ on $\neu$, and  $\partial_{\nu}u=ik \beta u+g$ on $\rob$, we obtain 
\begin{equation}
\begin{aligned}
&\int_{\partial \dmn\backslash \dir} |\nabla u|^2 m \cdot \nu ds -\int_{ \dir} |\partial_{\nu}u|^2 m \cdot \nu ds\\
&\qquad=\int_{\dmn} \mbox{div}(m)|\nabla u|^2dx- 2 \operatorname{Re}\int_{\dmn}\nabla u  \cdot( \nabla \bar{u}\nabla) m dx \\
&\qquad\qquad
 + 2 \operatorname{Re}\int_{\dmn}\frac{1}{A}(f+k^2V^2u+\nabla A \cdot \nabla u) (m \cdot \nabla \bar{u}) dx\\
&\qquad\qquad
 + 2 \operatorname{Re}\int_{\rob}(ik \beta  u+g ) (m \cdot \nabla \bar{u}) ds.
\end{aligned}
\end{equation}
Using \eqref{DONALD_BROWN:id1} with the transform $m\to \frac{V^2}{A} m$, we have
 \begin{align*}
&k^2\int_{\partial \dmn} |u|^2\left( \frac{V^2}{A}\right)m \cdot \nu ds
\\
&\qquad\qquad
=k^2\int_{\dmn} \mbox{div}\left(\frac{V^2}{A} m\right)|u|^2dx+ 2k^2 \operatorname{Re}\int_{\dmn}u \left(\frac{V^2}{A} \right)m \cdot \nabla \bar{u}dx.
 \end{align*}
Using this to replace the term $\operatorname{Re}\int_{\dmn}\left(\frac{V^2}{A}\right)u (m \cdot \nabla \bar{u}) dx$, we have 
 \begin{align*}
&\int_{\partial \dmn\backslash \dir} |\nabla u|^2 m \cdot \nu ds -\int_{ \dir} |\partial_{\nu}u|^2 m \cdot \nu ds\nonumber\\
&\quad=\int_{\dmn} \mbox{div}(m)|\nabla u|^2dx- 2 \operatorname{Re}\int_{\dmn}\nabla u  \cdot( \nabla \bar{u}\nabla) m dx\nonumber \\
&\qquad+ 2 \operatorname{Re}\int_{\dmn}\left(\frac{f}{A} \right)(m \cdot \nabla \bar{u}) dx+ 2 \operatorname{Re}\int_{\dmn}\left(\frac{\nabla A}{A}\right)\cdot\nabla u (m \cdot \nabla \bar{u}) dx \\
&\qquad+2 \operatorname{Re}\int_{\rob}(ik \beta u+g ) (m \cdot \nabla \bar{u}) ds\\
&\qquad-k^2  \int_{\dmn} \mbox{div}\left(\frac{V^2}{A} m\right)|u|^2dx+ k^2\int_{\partial \dmn} |u|^2 \left(\frac{V^2}{A}\right)m \cdot \nu ds.
\end{align*}
Expanding out the boundary terms in each of the portions we have 
\begin{equation}\label{DONALD_BROWN:identity1A}
 \begin{aligned}
&-\int_{ \dir} |\partial_{\nu}u|^2 m \cdot \nu ds +\int_{\neu} |\nabla u|^2 m \cdot \nu ds
\\
&\qquad+\int_{\rob} |\nabla u|^2 m \cdot \nu ds+k^2  \int_{\dmn} \mbox{div}\left(\frac{V^2}{A} m\right)|u|^2dx \\
&=\int_{\dmn} \mbox{div}(m)|\nabla u|^2dx- 2 \operatorname{Re}\int_{\dmn}\nabla u  \cdot( \nabla \bar{u}\nabla) m dx\\
&\qquad+ 2 \operatorname{Re}\int_{\dmn}\left(\frac{f}{A} \right)(m \cdot \nabla \bar{u}) dx+2 \operatorname{Re}\int_{\dmn}\left(\frac{\nabla A}{A}\right)\cdot\nabla u (m \cdot \nabla \bar{u}) dx \\
&\qquad+{k^2\int_{\neu} |u|^2 \left(\frac{V^2}{A}\right)m \cdot \nu ds+ k^2\int_{\rob} |u|^2 \left(\frac{V^2}{A}\right)m \cdot \nu ds}\\
&\qquad +   2 \operatorname{Re}\int_{\rob}(ik \beta  u+g ) (m \cdot \nabla \bar{u}) ds.
\end{aligned}
\end{equation}

Now we suppose we make the geometric assumptions made by \cite{DONALD_BROWN:hetmaniuk} outlined in \eqref{DONALD_BROWN:GeomAssum}. Recall, we have  for $m=x-x_{0}$, thus we compute
\begin{align*}
 \text{div}(x-x_{0})&=d \text{ in } \dmn,\\
 \nabla u  \cdot( \nabla \bar{u}\nabla) (x-x_{0})&=|\nabla u|^2 \text{ in } \dmn,\\
(x-x_{0})\cdot \nu &\leq 0 \text{ on } \dir,\\
(x-x_{0})\cdot \nu& = 0 \text{ on } \neu,\\
(x-x_{0})\cdot \nu &\geq \eta \text{ on } \rob.
\end{align*}
Using the above relations in \eqref{DONALD_BROWN:identity1A} we obtain 
\begin{equation}\label{DONALD_BROWN:ineq11}
 \begin{aligned}
&\eta \int_{\rob} |\nabla u|^2ds+k^2  \int_{\dmn} \mbox{div}\left(\frac{V^2}{A} (x-x_{0})\right)|u|^2dx \\
&\leq(d-2)\int_{\dmn} |\nabla u|^2dx+ 2 \operatorname{Re}\int_{\dmn}\left(\frac{f}{A} \right)((x-x_{0}) \cdot \nabla \bar{u}) dx \\
&+2 \operatorname{Re}\int_{\dmn}\left(\frac{\nabla A}{A}\right)\nabla u ((x-x_{0}) \cdot \nabla \bar{u}) dx \\
&+ k^2\int_{\rob} |u|^2 \left(\frac{V^2}{A}\right)(x-x_{0}) \cdot \nu ds+   2 \operatorname{Re}\int_{\rob}(ik \beta  u+g ) (m \cdot \nabla \bar{u}) ds.
\end{aligned}
\end{equation}
Recall, \eqref{DONALD_BROWN:Sfunction.theorem}, where we define the following function
\begin{equation}\label{DONALD_BROWN:Sfunction}
\begin{aligned}
S(x)&:=\mbox{div}\left(\left(\frac{V^2(x)}{A(x)}\right) (x-x_{0})\right) \\
&=d \left(\frac{V^2(x)}{A(x)}\right)+\left(2\frac{V(x) \nabla V(x)}{A(x)}-\frac{V^2(x) \nabla A(x)}{A^2(x)}\right)\cdot (x-x_{0}),
\end{aligned}
\end{equation}
and from \eqref{DONALD_BROWN:conditions}, we have a minimum for $S(x)$ exists and is positive 
$$S_{min}=\min_{x\in \dmn}S(x)>0.$$
Further, from \eqref{DONALD_BROWN:conditions}, we  have  $C_{G}$ to be the minimal constant so that 
\begin{align}
2 \abs{\int_{\dmn}\left(\frac{\nabla A}{A}\right)\nabla u ((x-x_{0}) \cdot \nabla \bar{u}) dx}\leq C_{G}\norm{\left(\frac{\nabla A}{A}\right)}_{L^{\infty}(\dmn)}\TwoNorm{\nabla u}{\dmn}^2.
\end{align}
Returning to inequality \eqref{DONALD_BROWN:ineq11}, we obtain 
\begin{equation}
 \begin{aligned}
&\eta \TwoNorm{\nabla u}{\rob}^2+k^2 S_{min} \TwoNorm{u}{\dmn}^2\\
&\leq (d- 2) \TwoNorm{\nabla u}{\dmn}^2+C_{G}\norm{\left(\frac{\nabla A}{A}\right)}_{L^{\infty}(\dmn)}\TwoNorm{\nabla u}{\dmn}^2\\
&\quad+C_{1} \left(\frac{1}{A_{min}}\TwoNorm{f}{\dmn}\TwoNorm{\nabla u}{\dmn}+\TwoNorm{g}{\rob}\TwoNorm{\nabla u}{\rob}\right) \\
&\quad+C_{1}\left( k^2\left(\frac{V_{max}^2}{A_{min}}\right)\TwoNorm{u}{\rob}^2+  k\norm{ \beta}_{L^{\infty}(\rob)} \TwoNorm{u}{\rob}\TwoNorm{\nabla u}{\rob}\right),
\end{aligned}
\end{equation}
where $C_{1}$ is independent of $k$ and the bounds \eqref{DONALD_BROWN:upperbounds}.
Note that on the right hand side we have
for any $\xi_5,\xi_6>0$  
the terms 
\begin{align*}
 k\norm{ \beta}_{L^{\infty}(\rob)} \TwoNorm{u}{\rob}\TwoNorm{\nabla u}{\rob}
&\leq \frac{k^2}{2 \xi_{5}} \TwoNorm{u}{\rob}^2+\frac{\xi_{5}}{2}\norm{ \beta}_{L^{\infty}(\rob)}^2\TwoNorm{\nabla u}{\rob}^2\\
  \TwoNorm{g}{\rob}\TwoNorm{\nabla u}{\rob}&\leq \frac{1}{2 \xi_{6}} \TwoNorm{g}{\rob}^2+\frac{\xi_{6}}{2}\TwoNorm{\nabla u}{\rob}^2.
\end{align*}
We choose $\xi_{5},\xi_{6}$ so that 
$$\frac{\eta}{2}=C_{1}\frac{\xi_{5}}{2}\norm{ \beta}_{L^{\infty}(\rob)}^2=C_{1}\frac{\xi_{6}}{2},$$
and so 
$$\frac{k^2}{2\xi_{5}}\leq \frac{C_{1}}{2\eta}\norm{ \beta}_{L^{\infty}(\rob)}^2k^2.$$
We then obtain 
\begin{equation}\label{DONALD_BROWN:ineq2}
 \begin{aligned}
k^2 S_{min} \TwoNorm{u}{\dmn}^2 &\leq C_{1}\left(\left(\frac{C_{1}}{2\eta}\norm{ \beta}_{L^{\infty}(\rob)}^2+\frac{V_{max}^2}{A_{min}}\right)k^2\TwoNorm{u}{\rob}^2\right)\\
&+ 
C_1
\left(\frac{1}{A_{min}}\TwoNorm{f}{\dmn}\TwoNorm{\nabla u}{\dmn}+\frac{C_{1}}{2\eta}\TwoNorm{g}{\rob}^2\right) 
\\&+(d- 2) \TwoNorm{\nabla u}{\dmn}^2+C_{G}\norm{\left(\frac{\nabla A}{A}\right)}_{L^{\infty}(\dmn)}\TwoNorm{\nabla u}{\dmn}^2.
\end{aligned}
\end{equation}
Taking $C^{bd}_{2}=C_{1}\left(\frac{C_{1}}{2\eta}\norm{ \beta}_{L^{\infty}(\rob)}^2+\frac{V_{max}^2}{A_{min}}\right)$ and letting $\epsilon=\beta_{min}\xi_{7}/C^{bd}_{2}$ in  the inequality \eqref{DONALD_BROWN:Lemmaest1} we have the relation 
\begin{align}
C_{2}^{bd}{k^2}  \TwoNorm{u}{\rob}^2 & \leq\frac{(C_{2}^{bd})^2}{\beta^2_{min}\xi_{7}}\TwoNorm{f}{\dmn}^2 +k^2\xi_{7} \TwoNorm{u}{\dmn}^2+\frac{C_{2}^{bd}}{\beta^2_{min}}\TwoNorm{g}{\rob}^2.
\end{align}
Applying this above inequality to \eqref{DONALD_BROWN:ineq2}, we obtain 
\begin{equation}\label{DONALD_BROWN:sminxi7est}
\begin{aligned}
&k^2 (S_{min}-\xi_{7}) \TwoNorm{u}{\dmn}^2 
\\&\qquad\leq 
C_1
\left(\frac{1}{A_{min}}\TwoNorm{f}{\dmn}\TwoNorm{\nabla u}{\dmn}+\frac{C_{1}}{2\eta}\TwoNorm{g}{\rob}^2\right)\\
&\qquad\qquad+\left((d- 2) +C_{G}\norm{\left(\frac{\nabla A}{A}\right)}_{L^{\infty}(\dmn)}\right)\TwoNorm{\nabla u}{\dmn}^2\\
& \qquad\qquad+\frac{(C_{2}^{bd})^2}{\beta^2_{min}\xi_{7}}\TwoNorm{f}{\dmn}^2 +\frac{C_{2}^{bd}}{\beta^2_{min}}\TwoNorm{g}{\rob}^2.
\end{aligned}
\end{equation}
Recall the estimate \eqref{DONALD_BROWN:Lemmaest2}, with $C_{3}^{bd}=\left((d- 2) +C_{G}\norm{\left(\frac{\nabla A}{A}\right)}_{L^{\infty}(\dmn)}\right)$, and taking $\xi_{4}=\frac{\xi_{3}}{2}=\xi_{8}$
\begin{align*}
&C_{3}^{bd}\TwoNorm{\nabla u}{\dmn}^2 \\
&\;\leq \frac{C_{3}^{bd}k^2}{A_{min}}\left( V_{max}^2+ \frac{\xi_{8}}{\beta_{min}}+\xi_{8}\right)\TwoNorm{u}{\dmn}^2\nonumber\\
&\quad+ \frac{C_{3}^{bd}}{A_{min}}\left(\frac{1}{4 k^2 \xi_{8}}+\frac{1}{\beta_{min}\xi_{8}} \right)\TwoNorm{f}{\dmn}^2+  \frac{C_{3}^{bd}}{A_{min}} \left(\frac{1}{\beta^2_{min}}+\frac{1}{4k^2}   \right)\TwoNorm{g}{\rob}^2.
\end{align*}
and so, using the above estimate 
\eqref{DONALD_BROWN:sminxi7est}%
we obtain
\begin{equation}\label{DONALD_BROWN:ineq3}
\begin{aligned}
&k^2 (S_{min}-\xi_{7}-\frac{C_{3}^{bd}}{A_{min}}\left( V_{max}^2+ \frac{\xi_{8}}{\beta_{min}}+\xi_{8}\right)) \TwoNorm{u}{\dmn}^2  \\
&\leq
C_1
\left(\frac{1}{A_{min}}\TwoNorm{f}{\dmn}\TwoNorm{\nabla u}{\dmn}+\frac{C_{1}}{2\eta}\TwoNorm{g}{\rob}^2\right)\\
&+ \frac{C_{3}^{bd}}{A_{min}}\left(\frac{1}{4 k^2 \xi_{8}}+\frac{1}{\beta_{min}\xi_{8}} \right)\TwoNorm{f}{\dmn}^2+  \frac{C_{3}^{bd}}{A_{min}} \left(\frac{1}{\beta^2_{min}}+\frac{1}{4k^2}   \right)\TwoNorm{g}{\rob}^2\\
& +\frac{(C_{2}^{bd})^2}{\beta^2_{min}\xi_{7}}\TwoNorm{f}{\dmn}^2 +\frac{C_{2}^{bd}}{\beta^2_{min}}\TwoNorm{g}{\rob}^2.
\end{aligned}
\end{equation}
Finally to deal with the 
remaining term on the right hand side that contains $\nabla u$,
we note using \eqref{DONALD_BROWN:Lemmaest2},
letting $\frac{\xi_{4}}{\beta_{min}}=\frac{\xi_{3}}{2}=\frac{V_{max}^2}{2}$, and multiplying by $\xi_{9}/(2A_{min}), \xi_9>0,$ we obtain 
\begin{align*}
&\frac{\xi_{9}}{2A_{min}}\TwoNorm{\nabla u}{\dmn}^2 \\
&\qquad\leq\frac{\xi_{9}}{2A^2_{min}}\bigg[2V_{max}^2 k^2\TwoNorm{u}{\dmn}^2+\left(\frac{2}{\beta^2_{min}V_{max}^2}
 +\frac{1}{
  2
  k^2 V_{max}^2} \right)\TwoNorm{f}{\dmn}^2
\\
&\qquad\qquad\qquad\qquad
 +\left(\frac{1}{\beta^2_{min}}+\frac{1}{4k^2}   \right)\TwoNorm{g}{\rob}^2\bigg],
\end{align*} 
and so 
\begin{align*}
&\frac{1}{A_{min}} \TwoNorm{f}{\dmn}\TwoNorm{\nabla u}{\dmn}\\
&\qquad
 \leq \frac{1}{2\xi_{9}A_{min}}\TwoNorm{f}{\dmn}^2+\frac{\xi_{9}}{2A_{min}}\TwoNorm{\nabla u}{\dmn}^2\\
&\qquad 
 \leq \frac{\xi_{9}V_{max}^2}{A_{min}^2}k^2 \TwoNorm{u}{\dmn}^2\\
&\qquad\qquad
 +\left(\frac{1}{2A_{min}\xi_{9}}+\frac{\xi_{9}}{2A_{min}^2}\left(\frac{2}{\beta^2_{min}V_{max}^2}
  +\frac{1}{
  2
  k^2 V_{max}^2} \right)\right)\TwoNorm{f}{\dmn}^2\\
&\qquad\qquad
+\frac{\xi_{9}}{2A_{min}^2}  \left(\frac{1}{\beta^2_{min}}+\frac{1}{4k^2}   \right)\TwoNorm{g}{\rob}^2 .
\end{align*}
Applying this into \eqref{DONALD_BROWN:ineq3}, we obtain
\begin{equation}\label{DONALD_BROWN:ineq4}
\begin{aligned}
&k^2 (S_{min}-\xi_{7}-\frac{C_{3}^{bd}}{A_{min}}\left( V_{max}^2+ \frac{\xi_{8}}{\beta_{min}}+\xi_{8}\right)
- \frac{
C_1
\xi_{9}V_{max}^2}{A_{min}^2}) \TwoNorm{u}{\dmn}^2 \\
&\leq
C_1
\left(\frac{1}{2A_{min}\xi_{9}}+\frac{\xi_{9}}{2A_{min}^2}\left(\frac{2}{\beta^2_{min}V_{max}^2}+\frac{1}{
2
k^2 V_{max}^2} \right)\right)\TwoNorm{f}{\dmn}^2\\
&+
C_1
\left(\frac{C_{1}}{2\eta}+\frac{\xi_{9}}{2A_{min}^2}  \left(\frac{1}{\beta^2_{min}}+\frac{1}{4k^2}   \right)\right)\TwoNorm{g}{\rob}^2\\
&
+ \frac{C_{3}^{bd}}{A_{min}}\left(\frac{1}{4 k^2 \xi_{8}}+\frac{1}{\beta_{min}\xi_{8}} \right)\TwoNorm{f}{\dmn}^2\\
&
+  \frac{C_{3}^{bd}}{A_{min}} \left(\frac{1}{\beta^2_{min}}+\frac{1}{4k^2}   \right)\TwoNorm{g}{\rob}^2
 +\frac{(C_{2}^{bd})^2}{\beta^2_{min}\xi_{7}}\TwoNorm{f}{\dmn}^2
+\frac{C_{2}^{bd}}{\beta^2_{min}}\TwoNorm{g}{\rob}^2.
\end{aligned}
\end{equation}
Hence, we see that the critical term is $S_{min}-\frac{C_{3}^{bd}V_{max}^2}{A_{min}}.$
Recall, $$C_{3}^{bd}:=\left((d- 2) +C_{G}\norm{\left(\frac{\nabla A}{A}\right)}_{L^{\infty}(\dmn)}\right),$$ thus, from \eqref{DONALD_BROWN:conditions}, we have
\begin{align}\label{DONALD_BROWN:finalcondition}
S_{min}-\left((d- 2) +C_{G}\norm{\left(\frac{\nabla A}{A}\right)}_{L^{\infty}(\dmn)}\right)\frac{V_{max}^2}{A_{min}}>0.
\end{align}

Since  \eqref{DONALD_BROWN:finalcondition} is assumed to hold, we take $\xi_7,\xi_8,$ and $\xi_9,$ so that 
$$
\left(S_{min}-\frac{C_{3}^{bd}V_{max}^2}{A_{min}}-\xi_{7}-\frac{C_{3}^{bd}\xi_{8}}{A_{min}}\left( \frac{1}{\beta_{min}}+1\right)- \frac{
C_1
\xi_{9}V_{max}^2}{A_{min}^2}\right)
>\delta
$$
for some $\delta>0$,
and taking 
$C_{4}^{bd}$
 to be the global constant bound for \eqref{DONALD_BROWN:ineq4} we obtain 
\begin{align}
 k^2\TwoNorm{ u}{\dmn}^2 
 \leq
  \frac{C_{4}^{bd}}{\delta}
 \left(1+\frac{1}{k^2}\right)\left(\TwoNorm{f}{\dmn}^2+\TwoNorm{g}{\rob}^2\right),
\end{align}
and using \eqref{DONALD_BROWN:Lemmaest2}, and taking 
$C_{5}^{bd}$
 to be the global constant bound we obtain 
\begin{align}
 \TwoNorm{\nabla u}{\dmn}^2 
 \leq
C_{5}^{bd}
 \left(1+\frac{1}{k^2}\right)\left(\TwoNorm{f}{\dmn}^2+\TwoNorm{g}{\rob}^2\right),
\end{align}
as desired. \qed
\end{proof}

\bibliographystyle{amsplain}
\bibliography{HetHelmholtz}
\end{document}